\documentclass{amsart}

\usepackage{amsmath,amsthm, amssymb,mathrsfs,  xypic}
\usepackage{quiver}
\usepackage{stackengine,scalerel}
\usepackage[colorlinks]{hyperref}
\usepackage[capitalize]{cleveref}

 \newcommand{\Nilp}{\mathrm{Nilp}}

\renewcommand{\det}{\mathrm{det}}

\newcommand{\Gal}{\mathrm{Gal}}

\newcommand{\Aut}{\mathrm{Aut}}

\newcommand{\Spf}{\mathrm{Spf}}
\newcommand{\Spec}{\mathrm{Spec}}
\newcommand{\Hom}{\mathrm{Hom}}

\newcommand{\Lie}{\mathrm{Lie}}

\newcommand{\GL}{\mathrm{GL}}

\newcommand{\et}{\mathrm{\acute{e}t}}

\newcommand{\cyc}{\mathrm{cyc}}

\newcommand{\ord}{\mathrm{ord}}
\newcommand{\Cont}{\mathrm{Cont}}

\newcommand{\Sym}{\mathrm{Sym}}

\newcommand{\Fl}{\mathcal{F}l}
\newcommand{\HT}{\mathrm{HT}}
\newcommand{\ab}{\mathrm{ab}}

\newcommand{\sm}{\mathrm{sm}}

\newcommand{\Katz}{\mathrm{Katz}}

\newcommand{\ur}{\mathrm{ur}}
\newcommand{\eval}{\mathrm{eval}}

\newcommand{\can}{\mathrm{can}}

\newcommand{\colim}{\mathrm{colim}}

\renewcommand{\Im}{\mathrm{Im}}

\newcommand{\bbQ}{\mathbb{Q}}

\newcommand{\bbZ}{\mathbb{Z}}

\newcommand{\Mant}{\mathrm{Mant}}

\newcommand{\Tate}{\mathrm{Tate}}

\newcommand{\Ig}{\mathfrak{I}}
\newcommand{\CS}{\mathrm{CS}}

\newcommand{\bdd}{\mathrm{bdd}}

\newcommand{\diag}{\mathrm{diag}}

\newcommand{\cusp}{\mathrm{cusp}}

\newcommand{\mr}[1]{\mathrm{#1}}
\newcommand{\ul}[1]{\underline{#1}}
\newcommand{\mbb}[1]{\mathbb{#1}}
\newcommand{\mc}[1]{\mathcal{#1}}

\newcommand{\Qpab}{\mathbb{Q}_p^{\mathrm{ab}}}

\newcommand{\Zpbreve}{\breve{\mathbb{Z}}_p}
\newcommand{\Qpbreve}{\breve{\mathbb{Q}}_p}
\newcommand{\Kir}{\mathrm{Kir}}

\numberwithin{equation}{subsubsection}
\theoremstyle{plain}

\newtheorem{maintheorem}{Theorem} 
 
\newtheorem{maincorollary}[maintheorem]{Corollary}

\newtheorem*{theorem*}{Theorem}
\newtheorem*{conjecture*}{Conjecture}
\newtheorem{theorem}[subsubsection]{Theorem}
\newtheorem{corollary}[subsubsection]{Corollary}
\newtheorem{conjecture}[subsubsection]{Conjecture}
\newtheorem{proposition}[subsubsection]{Proposition}
\newtheorem{lemma}[subsubsection]{Lemma}

\theoremstyle{definition}

\newtheorem{example}[subsubsection]{Example}
\newtheorem{definition}[subsubsection]{Definition}
\newtheorem{remark}[subsubsection]{Remark}

\crefformat{section}{\S#2#1#3}
\crefformat{Section}{\S#2#1#3}

\title[The completed Kirillov model]{The completed Kirillov model and local-global compatibility for functions on Igusa varieties}

\author{Sean Howe}
\date{\today}

\begin{document}

\begin{abstract}
We describe the cuspidal functions $\mathbb{V}_b^{\cusp}$ on the ordinary Caraiani-Scholze Igusa variety for $\GL_2$ as a completion of the smooth Kirillov model for classical cuspidal modular forms, and identify a variant of Hida's ordinary $p$-adic modular forms with the coinvariants of an action of $\tilde{\mu}_{p^\infty}$ on $\mathbb{V}_b^{\cusp}$. As a consequence of these results, we establish a weak local-global compatibility theorem for eigenspaces in $\mathbb{V}_b^{\cusp}$ associated to classical cuspidal modular forms. Based on these results, we conjecture an analog of Hida theory and an associated local-global compatibility for functions on more general Caraiani-Scholze Igusa varieties, which are natural spaces of $p$-adic automorphic forms. 
\end{abstract}

\maketitle

\setcounter{tocdepth}{1}
\tableofcontents

\section{Introduction}

Fix an integer prime $p$. In \cite{Howe.AUnipotentCircleActionOnpAdicModularForms}, we constructed an action of $\hat{\mathbb{G}}_m=\mu_{p^\infty}$ on the Katz-Serre space of $p$-adic modular forms integrating the differential operator $\theta=q\partial_q$. These $p$-adic modular forms are global functions on the Katz-Igusa formal scheme, and the construction was by descent of a larger action of $\tilde{\mu}_{p^\infty}=\mathcal{H}om(\mathbb{Q}_p, \mu_{p^\infty})$ on a natural cover, the Caraiani-Scholze Igusa formal scheme. As announced in \cite[1.2.13]{Howe.AUnipotentCircleActionOnpAdicModularForms} but deferred to the present work (see \cref{main.hida-coinvariants} and \cref{remark.ordinary-coinvariant-variants}), an elementary computation with $q$-expansions shows that the space of Hida-ordinary\footnote{We use this term to refer to ordinary $p$-adic modular forms in the sense of Hida to avoid confusion with references to the ordinary locus that pervade the theory of $p$-adic modular forms.} Katz-Serre $p$-adic modular forms can be identified with a space of topological co-invariants for the action of $\mu_{p^\infty}$ on Katz-Serre $p$-adic modular forms. 

In this work, we view cuspidal functions on the Caraiani-Scholze Igusa formal scheme as a natural space of $p$-adic \emph{automorphic} forms, and prove a local-global compatibility result at $p$ for this larger space, \cref{main.lg}. The interpretation of Hida-ordinary forms as coinvariants provided by \cref{main.hida-coinvariants} plays a key role in the proof. The other new ingredient is a construction of this space as a completion of the Kirillov model of the smooth representation of $\GL_2(\mathbb{A}_f)$ on classical cusp forms of weight $k \geq 2$, \cref{main.completion-kirillov-model}. The proof also appeals to Hida's finiteness for the ordinary Hecke algebra, and the theory of overconvergent companions as in \cite{BreuilEmerton.RepresentationPAdiquesOrdinarsDeGl2Qp}. 

 More generally, it is natural to consider functions on arbitrary Caraiani-Scholze Igusa formal schemes as spaces of $p$-adic automorphic forms. We conjecture (see \cref{ss.conj-gen} and \cref{s.conjecture}) that, for any such Igusa formal scheme, the natural space of co-invariants under a large unipotent group action is an admissible Banach representation of an associated $p$-adic Lie group. We also speculate on local-global compatibility in this setting. In particular, in \cref{ss.model-representations} we suggest a construction of local representations as measures on local Shimura varieties that is compatible with \cref{main.lg}.  

\begin{remark} We apologize to any interested readers for the long delay between the announcement of some of these results in \cite{Howe.AUnipotentCircleActionOnpAdicModularForms} and their appearance here! 
\end{remark}

\subsection{A completed Kirillov model}
We write $\mathbb{A}$ for the ad\`{e}les of $\mathbb{Q}$, $\mathbb{A}_f$ for the finite ad\`{e}les, and $\mathbb{A}_f^{(p)}$ for the prime-to-$p$ finite ad\`{e}les. 

For $b= \begin{bmatrix} p^{-1} & 0 \\ 0 & 1 \end{bmatrix}$, we let $\mathbb{V}_{b}^{\cusp}$ denote the cuspidal functions on the associated  Caraiani-Scholze Igusa formal scheme with infinite level away from $p$ (a cover of the ordinary locus of the formal $p$-adic modular curve). We let $\mathbb{V}_{b,\mbb{Q}_p}^{\cusp}=\mathbb{V}_{b}^{\cusp}[1/p]$, a $\mathbb{Q}_p$-Banach space equipped with commuting actions of $\GL_2(\mathbb{A}_f^{(p)})$ and a group $\tilde{G}_b$ that is an extension of $T(\mathbb{Q}_p)=\mathbb{Q}_p^\times \times \mathbb{Q}_p^\times$ by $\tilde{\mu}_{p^\infty}$. Our main results concern the structure of $\mathbb{V}_{b,\mbb{Q}_p}^{\cusp}$ as a Banach representation of these groups. We note that $\tilde{\mu}_{p^\infty}(\mathbb{Z}_p)=\{1\}$, but, by $p$-adic Fourier theory, the action of $\tilde{\mu}_{p^\infty}$ on $\mathbb{V}_{b,\mathbb{Q}_p}^\cusp$ can be understood as an action of the ring $C^\bdd(\mathbb{Q}_p, \mathbb{Q}_p)$ of bounded continuous functions on $\mathbb{Q}_p$. Thus, in place of the action of $\tilde{G}_b$, one can consider an action of $T(\mathbb{Q}_p) \ltimes C^{\bdd}(\mathbb{Q}_p, \mathbb{Q}_p).$

\begin{remark}\label{remark.p-adic-automorphic}
  The space $\mathbb{V}_{b,\mathbb{Q}_p}^{\mathrm{cusp}}$ is a natural space of $p$-adic \emph{automorphic} forms for $\GL_2$. Conceptually, this is because of the analogy, furnished by Kottwitz's theory of global isocrystals \cite{Kottwitz.BGForAllLocalAndGlobalFields}, between $\Ig^b$ and the archimedean space $\GL_2(\mathbb{Q})\backslash \GL_2(\mathbb{A})$ whose $\mathbb{C}$-valued functions host the classical theory of complex automorphic forms for $\GL_2$. At a practical level, this is true by comparison with more traditional spaces: as explained in \cite{Howe.AUnipotentCircleActionOnpAdicModularForms}, the Katz-Serre space of cuspidal $p$-adic \emph{modular} forms $\mathbb{V}_{\Katz,\mbb{Q}_p}^\cusp$ and its variant $\mathbb{V}_{\Mant, \mbb{Q}_p}^\cusp$, can be recovered as
  \[ \mathbb{V}_{\Katz,\mbb{Q}_p}^\cusp = \left( \mathbb{V}_{b,\mbb{Q}_p}^\cusp \right)^{(1 \times \mathbb{Z}_p^\times) \ltimes \mathbb{Z}_p(1)} \textrm{ and } \mathbb{V}_{\Mant, \mbb{Q}_p}^\cusp = \left( \mathbb{V}_{b, \mbb{Q}_p}^\cusp \right)^{\mathbb{Z}_p(1)},  \]
 where $\mathbb{Z}_p(1)=T_p(\mu_{p^\infty})$, the Tate module of $\mu_{p^\infty}$, is viewed as a subgroup of $\tilde{\mu}_{p^\infty}$.  
 \end{remark}

Our first result describes $\mathbb{V}_{b,\mathbb{Q}_p}^\cusp$ in terms of the smooth Kirillov models of classical cusp forms after an extension of scalars. For $\mathbb{C}_p=\overline{\mathbb{Q}_p}^\wedge$, we write $\mathbb{V}_{b,\mathbb{C}_p}^\cusp:=\mathbb{V}_{b,\mbb{Q}_p}^\cusp \hat{\otimes} \mathbb{C}_p$, which admits an action of $T(\mathbb{Q}_p) \ltimes C^{\bdd}(\mathbb{Q}_p, \mathbb{C}_p).$ We also fix a compatible system of primitive $n$th roots of unity in $\mathbb{C}_p$, which is used in some of the constructions below (e.g., in the additive character for smooth Kirillov models).  

For $k \in \mathbb{Z}$, we write $S_{k,\mathbb{C}_p}$ for the $\GL_2(\mathbb{A}_f)$-representation of classical cusp forms of weight $k$ over $\mathbb{C}_p$. We define natural evaluation maps 
\[ \eval_k: S_k \rightarrow \mathbb{V}_{b,\mathbb{C}_p}^{\cusp} \]
that are $\GL_2(\mathbb{A}_f^{(p)})$-equivariant and satisfy a twisted equivariance for upper triangular matrices in $\GL_2(\mathbb{Q}_p)$. Using $q$-expansions, we also define a natural embedding 
\[ \Kir: \mathbb{V}_{b,\mathbb{C}_p}^\cusp \rightarrow C^\bdd(\mathbb{A}_f^\times, \mathbb{C}_p) \]
that is equivariant for \[ \left((\mathbb{Q}_p^\times \times 1) \ltimes \tilde{\mu}_{p^\infty}\right) \times \begin{bmatrix}\mathbb{A}_f^{(p),\times} & \mathbb{A}_f^{(p)} \\ 0 & 1\end{bmatrix} \leq \tilde{G}_b \times \GL_2(\mathbb{A}_f^{(p)})\]
when $C^\bdd(\mathbb{A}_f^\times, \mathbb{C}_p)$ is equipped with a simple Kirillov style action --- in particular, the Fourier dual of the action of $\tilde{\mu}_{p^\infty}$ is simply the action of $C^\bdd(\mathbb{Q}_p, \mathbb{C}_p)$ through multiplication by functions pulled back along the projection $\mathbb{A}_f^\times \rightarrow \mathbb{Q}_p^\times \subseteq \mathbb{Q}_p$. 

\begin{maintheorem}\label{main.completion-kirillov-model}(\cref{theorem.completion-kirillov-model})
For any $k \in \mathbb{Z}$, $\Kir \circ \eval_k$ is the smooth Kirillov model of $S_{k,\mathbb{C}_p}$. For  $k \geq 2$, $\Kir \circ \eval_k$ identifies $\mathbb{V}_{b, \mathbb{C}_p}^\cusp$ with the completion of the smooth Kirillov model of $S_{k,\mathbb{C}_p}$ inside $C^\bdd(\mathbb{A}_f^\times, \mathbb{C}_p)$ (for the supremum norm). 	
\end{maintheorem}

After constructing the map and checking the equivariance, there are two ingredients in the proof of \cref{main.completion-kirillov-model}: 
The density of $S_{k,\mathbb{C}_p}$, $k \geq 2$ in $\mathbb{V}_{b, \Qpab}^\cusp$ is an analog of the classical density of cusp forms of level $\Gamma_1(p^n)$ in $\mathbb{V}_{\Katz,\mbb{Q}_p}^\cusp$ (attributed by Hida to Shimura), and the proof of the former is by bootstrap from the latter. That $\Kir$ is a topological embedding is an analog of Katz's $q$-expansion principle for $\mathbb{V}_{\Katz,\mbb{Q}_p}^\cusp$, and the proof of the former is also essentially by bootstrap from the latter. 

\begin{remark}  
If we instead consider $\bigoplus_{k \in \mathbb{Z}} \eval_k$, then one can obtain the density by bootstrap from Katz's density result for modular forms of level $1$ at $p$ but arbitrary weight \cite{Katz.HigherCongruencesBetweenModularForms}. In \cite{Howe.ThepAdicJacquetLanglandsCorrespondenceAndAQuestionOfSerre} we gave a different argument to show the image of $\bigoplus_{k \in \mathbb{Z}} \eval_k$ is dense that uses more recent tools but is very soft in the sense that it can be applied to establish similar statements for other Igusa varieties. 
\end{remark}

\subsection{Local-global compatibility}

Suppose $\pi=\pi_p \otimes \pi^{(p)} \subseteq S_{k,\mathbb{C}_p}$ is an irreducible representation of $\GL_2(\mathbb{A}_f)=\GL_2(\mathbb{Q}_p)\times \GL_2(\mathbb{A}_f^{(p)})$. It is natural to try to describe 
\[ W_\pi:= \Hom_{\GL_2(\mathbb{A}_f^{(p)})}(\pi^{(p)}, \mathbb{V}_{b,\mathbb{C}_p}^\cusp) \]
as a representation of $T(\mathbb{Q}_p) \times C^\bdd(\mathbb{Q}_p, \mathbb{C}_p)$. As a first step, combining \cref{main.completion-kirillov-model} with the theory of new vectors gives rise to canonical  embeddings 
\begin{equation}\label{eq.local-global-inclusions} (\pi_p^\Kir)^\wedge \subseteq W_\pi \subseteq C^\bdd(\mathbb{Q}_p^\times, \mathbb{C}_p)  \end{equation}
where $(\pi_p^\Kir)^\wedge$ denotes the completion of the Kirillov model for $\pi_p$ in $C^\bdd(\mathbb{Q}_p^\times, \mathbb{C}_p)$.

Attached to $\pi$, there is an irreducible Galois representation\footnote{The precise normalizations are given in \cref{ss.galois-classical-modular-forms}, but we summarize them here: we view modular forms of weight $k$ as sections of $\omega_{E^\vee}^k$ to obtain the $\GL_2(\mathbb{A}_f)$ action, and then the trace of geometric Frobenius for $\rho$ at an unramified prime $\ell$ is given by $\frac{1}{\ell}$ times the eigenvalue of the double coset operator $\GL_2(\mathbb{Z}_\ell)\begin{bmatrix}\ell & 0 \\ 0 & 1 \end{bmatrix} \GL_2(\mathbb{Z}_\ell)$ acting on  $\pi_\ell^{\GL_2(\mathbb{Z}_\ell)}$.}  $\rho: \Gal(\overline{\mathbb{Q}}/\mathbb{Q}) \rightarrow \GL_2(\mathbb{C}_p)$ and one expects that $W_\pi$ depends only on the restriction $\rho_p= \rho|_{\Gal(\overline{\mathbb{Q}}_p/\mathbb{Q}_p)}$. The results of \cite{JacquetLanglands.AutomorphicFormsOnGL2} give an explicit description of $\pi_p^\Kir$ that we can use to compute $(\pi_p^\Kir)^\wedge$ so that its contribution can be made explicit. In particular, the following conjecture says that if $\rho_p$ is indecomposable then we expect the first inclusion in \cref{eq.local-global-inclusions} to be an isomorphism. We write $\mc{S}(\mathbb{Q}_p^\times, \mathbb{C}_p) \subseteq C^\bdd(\mathbb{Q}_p, \mathbb{C}_p)$ for the continuous $\mathbb{C}_p$-valued functions on $\mathbb{Q}_p^\times$ converging to $0$ at both $0$ and $\infty$. We write $1_{\mathbb{Z}_p}$ for the indicator function of $\mathbb{Z}_p$ in $\mathbb{Q}_p$, or its restriction to $\mathbb{Q}_p^\times$. 

\begin{conjecture}\label{conj.lg-mf}Suppose $k \geq 2$ and $\pi=\pi_p \otimes \pi^{(p)} \subseteq S_{k,\mathbb{C}_p}$ is an irreducible representation of $\GL_2(\mathbb{A}_f)$. Let  $W_\pi:=\Hom_{\GL_2(\mathbb{A}_f^{(p)})}(\pi^{(p)}, \mathbb{V}_{b,\mathbb{C}_p}^\cusp)$. Then, as a representation of $(\mathbb{Q}_p^\times \times 1) \ltimes C^\bdd(\mathbb{Q}_p, \mathbb{C}_p)$, 
 \[   W_\pi = \mc{S}(\mathbb{Q}_p^\times, \mathbb{C}_p) + \sum_{\chi \subset \rho_p} \mathbb{C}_p \cdot (1_{\mathbb{Z}_p} \cdot \chi) \subseteq C^{\bdd}(\mathbb{Q}_p^\times, \mathbb{C}_p) \]
 where in the sum $\chi$ runs over the (at most $2$) characters appearing as subrepresentations of $\rho_p$, and each such $\chi$ is viewed as a character of $\mathbb{Q}_p^\times$ by class field theory (normalized so that uniformizers match with geometric Frobenius).  
\end{conjecture}

\begin{remark}\label{remark.central-action} The action of the diagonal $\mathbb{Q}_p^\times$ in $T(\mathbb{Q}_p)$ on  $W_\pi$ is through a character determined  by the weight $k$ and the central character of $\pi_p$ (see \cref{lemma.central-character}), so that \cref{conj.lg-mf} describes the full $T(\mathbb{Q}_p) \ltimes C^\bdd(\mathbb{Q}_p, \mathbb{C}_p)$ action on $W_\pi$. 
\end{remark}

\begin{maintheorem}\label{main.lg}In the notation of \cref{conj.lg-mf}, there exists $M \geq 0$ such that
\[ \mc{S}(\mathbb{Q}_p^\times, \mathbb{C}_p) + \sum_{\chi \subset \rho_p} \mathbb{C}_p \cdot (1_{\mathbb{Z}_p} \cdot \chi) \subseteq W_\pi \subseteq \mc{S}(\mathbb{Q}_p^\times, \mathbb{C}_p) + \sum_{\chi \subset \rho_p} \sum_{a,b \leq M} \mathbb{C}_p \cdot 1_{\mathbb{Z}_p} \cdot \chi_{a,b} \]
where, writing $\mathbb{Q}_p^\times= p^\mathbb{Z} \times \mu(\mathbb{Q}_p) \times (1+2p\mathbb{Z}_p)$, 
\[ \chi_{a,b}(p^k, \zeta, t)=\chi(p)^k k^a \chi(\zeta) \chi(t)\log(t)^b. \] 
In particular, \cref{conj.lg-mf} holds if $\rho_p$ is irreducible (the non-ordinary case). 
\end{maintheorem}

\subsubsection{}We sketch the argument for \cref{main.lg}: first, $\pi_p^\Kir$ always contains the locally constant compactly supported function on $\mathbb{Q}_p^\times$, and thus its closure contains a copy of $\mc{S}(\mathbb{Q}_p^\times, \mathbb{C}_p)$. The codimension of these compactly supported functions in $\pi_p^\Kir$ is $0$ if $\pi_p$ is supercuspidal, $1$ if $\pi_p$ is special, and $2$ if $\pi_p$ is principal series --- this is the dimension of the Jacquet module $J(\pi_p)$, and the extra dimensions that appear can be represented by the functions $1_{\mathbb{Z}_p} \cdot \epsilon(z,1)$ for $\epsilon$ a character of $T(\mathbb{Q}_p)$ appearing in $J(\pi_p)$. With our normalizations, $|\epsilon(p,1)| \leq 1$, with equality for at most one $\epsilon$. In particular, when the inequality is strict the function $1_{\mathbb{Z}_p} \cdot \epsilon(z,1)$ is already in $\mc{S}(\mathbb{Q}_p^\times, \mathbb{C}_p)$ so that we obtain nothing new; it is only in the ordinary case when $|\epsilon(p,1)|=1$ that we obtain something new, and in that case $\chi=\epsilon|_{\mathbb{Q}_p^\times \times 1}$ appears also as a subrepresentation of $\rho_p$. Thus we obtain a computation of $(\pi_p^\Kir)^\wedge$ and, when $\rho$ is indecomposable, it accounts for everything in the leftmost side of the containments in \cref{main.lg}. When $\rho_p$ is a direct sum of two characters, we obtain the missing vector as a twist of the ordinary Breuil-Emerton \cite{BreuilEmerton.RepresentationPAdiquesOrdinarsDeGl2Qp} companion.

To bound the size of the representation from above, we argue as follows: first, using the action of $C^\bdd(\mathbb{Q}_p, \mathbb{C}_p)$, we may view $W_\pi$ as a sheaf on $\mathbb{Q}_p$, and then pass to the fiber at zero. This is a version of topological co-invariants for the $\tilde{\mu}_{p^\infty}$ action, which we write as $(W_\pi)_{\tilde{\mu}_{p^\infty}}$. To understand these, we show that the topological co-invariants for $\mathbb{V}_b^\cusp$ are equivalent to a variant of the Hida-ordinary $p$-adic modular forms (\cref{main.hida-coinvariants}). In particular, Hida's finiteness theorem for the ordinary Hecke algebra implies that, since we already know the central character, the action of $\mathbb{Z}_p^\times \times 1$ on $(W_\pi)_{\tilde{\mu}_{p^\infty}}$ gives an admissible Banach representation of $\mathbb{Z}_p^\times$. Using the structure of these admissible representations, and a consideration of the Galois representations associated to ordinary $p$-adic modular forms, we find that the action of $\mathbb{Q}_p^\times$ is supported at the characters $\chi$ appearing in \cref{main.lg}. The vectors $1_{\mathbb{Z}_p} \cdot \chi_{a,b}$ for $a$ and/or $b$ greater than zero correspond to the possibility of generalized eigenspaces, which we do not see a way to exclude by our argument. 

\begin{remark}\label{remark.lg-equivalences}
   For $\rho_p$ irreducible,  \cref{main.lg} implies that knowing $W_\pi$ 
   is equivalent to knowing $\det(\rho_p)$ (using the central character).   When $\rho_p$ is reducible and potentially crystalline (i.e. $\pi_p$ is ordinary principal series), \cref{main.lg} implies $W_\pi$ (or $(W_\pi)_{\tilde{\mu}_{p^\infty}}$) determines $\rho_p$. When $\rho_p$ is reducible but potentially semistable with non-trivial monodromy operator (i.e. $\pi_p$ is ordinary Steinberg), \cref{main.lg} implies $W_\pi$ (or $(W_\pi)_{\tilde{\mu}_{p^\infty}}$) determines $\rho_p^\mathrm{ss}$.    In both reducible cases, if \cref{conj.lg-mf} holds then $W_\pi$ is also determined by $\rho_p$.  
\end{remark}

\subsection{Hida theory and coinvariants}
Recall that $\mathbb{V}_b^{\cusp}$ is the unit ball in $\mathbb{V}_{b,\mathbb{Q}_p}^\cusp$; we write also $\mathbb{V}_{\Mant}^\cusp=\mathbb{V}_b^{\cusp, \mathbb{Z}_p(1)}$ (cf. \cref{remark.p-adic-automorphic}). There is a natural $U_p$-operator defined on $\mathbb{V}_\Mant^\cusp$ and extending the usual $U_p$-operator in the theory of $p$-adic modular forms. The result linking Hida theory and coinvariants, which plays a key role in the proof of \cref{main.lg} described above, is:

\begin{maintheorem}[\cref{theorem.ordinary-iso-coinvariants}]\label{main.hida-coinvariants}
    For $e = \lim_n {U_p}^{n!}$ the ordinary projector on $\mathbb{V}_{\Mant}^\cusp$, restriction $e\mathbb{V}_{\Mant}^\cusp \rightarrow (\mathbb{V}_{b}^\cusp)_{\tilde{\mu}_{p^\infty}}$ is a $T(\mathbb{Q}_p) \times \GL_2(\mathbb{A}_f^{(p)})$-equivariant isomorphism. 
\end{maintheorem}

Here the $U_p$ operator acts on $q$-expansions in the usual way: one finds that the map $\Kir$ on $\mathbb{V}_{\Mant}^
\cusp$ factors through functions on $(\mathbb{Z}_p\backslash 0) \times (\mathbb{A}_f^{(p)})^\times$, and, in this Kirillov model, $(U_p f)(z_p z^{(p)})=f( (pz_p) z^{(p)})$. From this description, it is easy to see that the elements in the kernel of $e$ are exactly those killed by passing to the fiber at $0$ in $\mathbb{Q}_p$, i.e. to  $(\mathbb{V}_{b}^\cusp)_{\tilde{\mu}_{p^\infty}}$. Thus the proof of \cref{main.hida-coinvariants} is an elementary argument with $q$-expansions, so that the main novelty is the statement itself. 

\begin{remark}\label{remark.ordinary-coinvariant-variants}
We have $\mathbb{V}_\Mant^\cusp= 1_{\mathbb{Z}_p} \cdot \mathbb{V}_b^\cusp$ (the action of the $1_{\mathbb{Z}_p}$ is a projection onto the $\mathbb{Z}_p(1)$-invariants by a normalized trace). As in \cite{Howe.AUnipotentCircleActionOnpAdicModularForms}, there is a residual action of $\mu_{p^\infty}=\tilde{\mu}_{p^\infty}/\mathbb{Z}_p(1)$ on $\mathbb{V}_\Mant^\cusp$, whose Fourier dual is an action of $C(\mathbb{Z}_p, \mathbb{Z}_p) = 1_{\mathbb{Z}_p} \cdot C(\mathbb{Q}_p, \mathbb{Z}_p)$. Since taking the fiber at $0 \in \mathbb{Q}_p$ factors through this projection followed by taking the fiber at $0 \in \mathbb{Z}_p$, \cref{main.hida-coinvariants} is equivalent to
\[ e\mathbb{V}_{\Mant}^\cusp = (\mathbb{V}_{b}^\cusp)_{\tilde{\mu}_{p^\infty}} = (\mathbb{V}_{\Mant}^\cusp)_{\mu_{p^\infty}}. \]
Similarly, if we work on $\mathbb{V}_{\Katz}^\cusp$ as in \cite{Howe.AUnipotentCircleActionOnpAdicModularForms} (by viewing the Katz-Igusa formal scheme as a connected component of the Mantovan-Igusa formal scheme), then we can show
\[e \mathbb{V}_{\Katz}^\cusp = (\mathbb{V}_{\Katz}^\cusp)_{\mu_{p^\infty}} \]
where the left-hand side is now on the nose the space of ordinary forms considered by Hida (at least after passing to invariants for a compact open $K^p \leq \GL_2(\mathbb{A}_f^{(p)})$). 
\end{remark}

 Combining \cref{main.hida-coinvariants} with Hida's finiteness theorem for the ordinary Hecke algebra (interpreted dually as an admissibility statement for $e \mathbb{V}_{\Katz}^\cusp$), we show
 \begin{maincorollary}\label{maincor.admissibility} For $K^p \leq \GL_2(\mathbb{A}_f^{(p)})$ compact open,  $(\mathbb{V}_{b, \mbb{Q}_p}^{\cusp,K^p})_{\tilde{\mu}_{p^\infty}}$ is an admissible Banach representation of $T(\mathbb{Q}_p)$.
 \end{maincorollary}

\begin{remark}
   \Cref{main.lg} should be compared with the usual descriptions of the Jacquet module and Kirillov models of smooth generic irreducible representations of $\GL_2(\mathbb{Q}_p)$ (see \cref{ss.jacquet-kirillov}), \cref{main.hida-coinvariants} should be compared with the theory of the smooth canonical lift, and \cref{maincor.admissibility} should be compared to the admissibility of the Jacquet module of the smooth admissible $\GL_2(\mathbb{Q}_p)$-representation $S_k^{K^p}$. 
\end{remark}

\subsection{A conjectural generalization}\label{ss.conj-gen}
In \cref{s.conjecture}, we conjecture a generalization of \cref{maincor.admissibility} to analogous spaces of topological invariants for functions on the general Caraiani-Scholze Igusa varieties of \cite{CaraianiScholze.OnTheGenericPartOfTheCohomologyOfCompactUnitaryShimuraVarieties} (for simplicity, we are explicit only in the case of compact Shimura varieties). These spaces are associated to a Shimura datum $(G,X)$ and an element $b \in G(\Qpbreve)$ satisfying a compatibility condition with the Hodge cocharacter; we write $\Ig^b$ for the affine $p$-adic formal scheme associated to such a datum. There is a unipotent group $\tilde{U}_b$ which acts on $\Ig^b$, playing the role of $\tilde{\mu}_{p^\infty}$, and we define a notion of topological coinvariants for the $\tilde{U}_b$-action on the functions $\mathcal{O}(\Ig_b)$. Our conjecture is that, for any compact open subgroup $K^p \leq G(\mathbb{A}_f^{(p)})$,  $\mathcal{O}(\Ig_b)_{\tilde{U}_b}[1/p]^{K^p}$ is an admissible Banach representation of the associated group $G_b(\mathbb{Q}_p)$ (an inner form of a Levi subgroup of $G$) --- see \cref{conj.precise-coinvariants}. There is further evidence for this conjecture beyond \cref{maincor.admissibility} for the basic and $(\mu-)$ordinary Newton strata, i.e. in the two extremal cases (see \cref{remark.mu-ord-and-basic}). 

Beyond the admissibility conjecture, we also speculate on how the local-global compatibility in this putative admissible part may generalize from \cref{main.lg} and \cref{conj.lg-mf}. In particular, in \cref{ss.model-representations} we explain a local construction of $G_b(\mbb{Q}_p)$-representations as spaces of measures on local Shimura varieties that we expect to appear in such a local-global compatibility statement for regular cohomological automorphic forms as discussed in \cref{ss.conj-local-global-compatibility}; \cref{example.model-construction} shows this local construction recovers in a natural way the characters appearing in \cref{main.lg}.

\subsection{Outline}
In \cref{s.preliminaries} we recall various preliminaries. In \cref{s.rep-theory} we discuss the representation theory of $\tilde{\mu}_{p^\infty}, \mathbb{Z}_p(1),$ and $\mu_{p^\infty}$ with an emphasis on a dual interpretations via $p$-adic Fourier theory and its role in defining and studying coinvariants --- one interesting aspect of this study is that certain properties that at first may seem specific to our geometric setup are in fact general properties of such representations. 

In \cref{s.ordinary-igusa-varieties} we recall the definitions of various $\GL_2$ ordinary Igusa varieties and their group actions, as well as the associated function spaces appearing in our main results. The results of \cref{s.rep-theory} are used here to organize some of the local actions at $p$. 

In \cref{s.completed-kirillov} we set up and prove \cref{main.completion-kirillov-model} on the completed Kirillov model. In \cref{s.coinvariants-ordinary-lg} we make the link with Hida theory by proving \cref{main.hida-coinvariants}, then use this to deduce \cref{maincor.admissibility} on the admissibility of the coinvariants. In \cref{s.local-global-proof} we combine the previous results to give the proof of our local-global compatibility result, \cref{main.lg}.

In \cref{s.conjecture} we first make our general admissibility conjecture precise as \cref{conj.precise-coinvariants} then discuss complements on a general local-global compatibility conjecture. 

\subsection{Acknowledgements} We thank Andrew Graham and Pol van Hoften for helpful conversations related to \cref{conj.precise-coinvariants}. We also thank Pol van Hoften for helpful comments and suggestions on an earlier version and, in particular, for catching some mistakes in an earlier version of \cref{s.conjecture}! We thank Michael Harris and Nike Vatsal for encouraging us to resurrect this work. Sean Howe was supported by the National Science Foundation through grants DMS-2201112 and DMS-1704005.

\section{Preliminaries, notation, and normalizations}\label{s.preliminaries}

\subsection{Notation}
\subsubsection{Ad\`{e}les}
Let $\hat{\mathbb{Z}}=\lim_n \mathbb{Z}/n\mathbb{Z}$ be the profinite completion of $\mathbb{Z}$ and let $\mathbb{A}_f=\hat{\mathbb{Z}}\otimes_{\mathbb{Z}} \mathbb{Q}$ denote the finite ad\`{e}les of $\mathbb{Q}$. For $p$ a rational prime, let $\hat{\mathbb{Z}}^{(p)}=\lim_{(n,p)=1}\mathbb{Z}/n\mathbb{Z}$ and let $\mathbb{A}_f^{(p)}=\hat{\mathbb{Z}}^{(p)} \otimes_{\mathbb{Z}} \mathbb{Q}$ denote the prime-to-$p$ adeles. 

\subsubsection{$\GL_2$ and its subgroups}
We write $\GL_2$ for the group scheme over $\mathbb{Z}$ of $2 \times 2$ invertible matrices. We write $B \leq \GL_2$ for the subgroup scheme of invertible upper triangular matrices, $U \leq \GL_2$ for the subgroup scheme of upper triangular unipotent matrices, and $T \leq \GL_2$ for the subgroup scheme of invertible diagonal matrices. We write $B_1 \leq B$ for the mirabolic subgroup scheme consisting of matrices in $B$ with $1$ in the lower right entry. For $R$ a ring and $r,s\in R^\times$, we write 
\[ \diag(r,s):= \begin{bmatrix} r & 0 \\ 0 & s\end{bmatrix} \in T(R). \]

\subsubsection{Congruence subgroups}\label{sss.congruence-subgroups}
For $N \in \mathbb{Z}_{\geq 1}$, we write $K_1(N)$ for the subgroup of $\GL_2(\hat{\mathbb{Z}})$ consisting of matrices whose image in $\GL_2(\mathbb{Z}/N\mathbb{Z})$ lies in $B_1(\mathbb{Z}/N\mathbb{Z})$. For $p$ a rational prime and $N$ coprime to $p$, we write $K_1^{(p)}(N)$ for the subgroup of $\GL_2(\hat{\mathbb{Z}}^{(p)})$ consisting of matrices whose image in $\GL_2(\mathbb{Z}/N\mathbb{Z})$ lies in $B_1(\mathbb{Z}/N\mathbb{Z})$.

\subsubsection{Local class field theory}\label{sss.lcft}
For $\ell$ a rational prime, we normalize the isomorphism
$\mbb{Q}_\ell^\times \cong W_{\mbb{Q}_\ell}^{\mr{ab}}$
 of local class field theory so that the class of uniformizers $\ell\mathbb{Z}_\ell^\times \in \mathbb{Q}_\ell^\times/\mathbb{Z}_\ell^\times$ is mapped to the geometric Frobenius element in $\Gal(\overline{\mathbb{F}}_\ell/\mathbb{F}_\ell)$ (i.e. the inverse of the $\ell$-power arithmetic Frobenius automorphism $x \mapsto x^\ell$ of $\overline{\mathbb{F}}_\ell$). 

\subsection{Jacquet modules and Kirillov models}\label{ss.jacquet-kirillov}
Let $\ell$ be a rational prime and   let $C/\mathbb{Q}$ be algebraically closed. We recall here some standard results on the smooth representation theory of $\GL_2(\mbb{Q}_\ell)$ over $C$ (see, e.g., \cite[Chapter IV]{Bump.AutomorphicFormsAndRepresentations}).

We write $\delta$ for the modulus character of $B(\mbb{Q}_\ell)$ (which factors through $T(\mbb{Q}_\ell)$). For $v_\ell$ the $\ell$-adic valuation and  $|\cdot|=\ell^{-v_\ell(\cdot)}$ the usual $\ell$-adic absolute value,
\[ \delta\left( \begin{bmatrix} a & u\\ 0 & d \end{bmatrix} \right) = |a| |d|^{-1}. \]
We fix a squareroot of $\ell$ in $C$ in order to define $\delta^{1/2}$; the statements below do not depend on this choice. 

For $\chi_1$ and $\chi_2$ two characters $\mbb{Q}_{\ell}^\times \rightarrow C^\times$, we let $\epsilon_{1,2}: T(\mbb{Q}_p) \rightarrow C^\times$ be the character $\begin{bmatrix} a & 0 \\ 0 & d \end{bmatrix} \mapsto \chi_1(a)\chi_2(d)$, and we let $\mc{B}(\chi_1,\chi_2)$ be the normalized parabolic induction of $\epsilon_{1,2}$ from $B(\mbb{Q}_\ell)$ to $\GL_2(\mbb{Q}_\ell)$, i.e. the smooth parabolic induction of $\epsilon_{1,2}\delta^{1/2}$, i.e. the subspace of $C^\sm(\GL_2(\mbb{Q}_\ell), C)$ consisting of locally constant functions $f$ such that
\[ f\left(\begin{bmatrix}a & u \\ 0 & d\end{bmatrix} x\right) = \chi_1(a)\chi_2(d)|a|^{1/2}|d|^{-1/2} f(x), \]
equipped with the left $\GL_2(\mbb{Q}_\ell)$-action by right translation, $(g \cdot f)(x)=f(xg)$. 

We recall that $\mc{B}(\chi_1, \chi_2)$ is irreducible if and only $\chi_1|\cdot|^{\pm 1} \neq \chi_2$, and that in this case $\mc{B}(\chi_1,\chi_2)=\mc{B}(\chi_2,\chi_1)$; these representations are called the irreducible principal series. The representation $\mc{B}(\chi|\cdot |^{-1}, \chi)$ has a one-dimensional subrepresentation, whose associated character is $|\cdot|^{-1/2}\chi \circ \det$, and the quotient is irreducible; such quotients are called special. An infinite dimensional irreducible representation that is neither irreducible principal series or special is called supercuspidal. 

If $\pi$ is a smooth representation of $\GL_2(\mbb{Q}_\ell)$ on a $C$-vector space, we write $J(\pi)$ for the Jacquet module of $\pi$, i.e. for the $U(\mbb{Q}_\ell)$-coinvariants $\pi_{U(\mathbb{Q}_\ell)}$ viewed as a $T(\mbb{Q}_\ell)$ representation. If $\pi$ is admissible, then $J(\pi)$ is admissible by the theory of the canonical lift as in \cite[Proposition 4.1.4]{Casselman.OnRepresentationsOfGL2AndTheArithmeticOfModularCurves}.  

We now recall more specific results on the Jacquet modules of irreducible representations and their relation with Kirillov models (contained in, e.g., \cite[Chapter IV]{Bump.AutomorphicFormsAndRepresentations}). We write $C^\sm(\mathbb{Q}_p^\times, C)$ for the locally constant  functions on $\mathbb{Q}_p^\times$ with values in $C$ and $C_c^{\sm}(\mathbb{Q}_p^\times, C)$ for the subspace of compactly supported functions. 

\begin{theorem}\label{theorem.jaquetandkirillov}
    Suppose $\pi$ is a smooth irreducible representation of $\GL_2(\mbb{Q}_\ell)$. Then $\pi$ is admissible. Moreover,
    \begin{enumerate}
        \item If $\pi$ is finite dimensional, then $\pi$ is one-dimensional and the associated character can be written as $\chi \circ \det$ where $\chi: \mbb{Q}_\ell^\times \rightarrow C^\times$ is a character and $\det: \GL_2(\mbb{Q}_\ell) \rightarrow \mbb{Q}_\ell^\times$ is the determinant.
        \item If $\pi$ is infinite dimensional, then, for $\tau$ a non-trivial character of $U(\mbb{Q}_\ell)$ (which is canonically identified with the additive group $\mbb{Q}_\ell$), 
        \begin{enumerate}
        \item $\pi$ admits a unique \emph{Kirillov model} as a representation of $\GL_2(\mbb{Q}_\ell)$ on a subspace  
        \[ \pi^\Kir \subseteq C^{\sm}(\mbb{Q}_\ell^\times, C) \]
        such that the action of the mirabolic $B_1(\mbb{Q}_\ell)$ on $\pi^\Kir$ is described by 
        \[ \left( \begin{bmatrix} a & u \\ 0 & 1 \end{bmatrix} \cdot f\right) (x) = \tau(ux) f(ax). \]
        \item $C^{\sm}_c(\mbb{Q}_\ell^\times, C) \subseteq \pi^{\Kir},$ and $J(\pi)=\pi^\Kir/C^{\sm}_c(\mbb{Q}_\ell^\times, C)$.
        \item $\dim J(\pi)=2$ if $\pi$ is irreducible principal series, $\dim J(\pi)=1$ if $\pi$ is special, and $\dim J(\pi)=0$ if $\pi$ is supercuspidal. 
        \end{enumerate}
    \end{enumerate}
\end{theorem}

\begin{example}\label{example.kirillov-models} In particular, we find that an irreducible $\pi$ is supercuspidal if and only if $\pi^\Kir=C_{c}^\sm(\mbb{Q}_\ell^\times, C)$.  In the other cases, we can also describe $J(\pi)$ and $\pi^\Kir$ explicitly (see \cite[Theorems 4.7.2-4.7.3]{Bump.AutomorphicFormsAndRepresentations}). Below we write $1_{\mbb{Z}_\ell}$ for the indicator function of $\mbb{Z}_\ell$ in $\mbb{Q}_\ell$ (or rather, by abuse of notation, its restriction to $\mbb{Q}_\ell^\times$). 
\begin{enumerate}
\item Suppose $\pi=\mc{B}(\chi_1, \chi_2)$  and $\chi_1 \neq \chi_2$. Then $J(\pi)$ is the direct sum of 
\begin{align*} \epsilon=\epsilon_{1,2} \delta^{1/2}: \begin{bmatrix}a & 0 \\ 0 & d\end{bmatrix} &\mapsto \chi_1(a)|a|^{1/2} \chi_2(b)|b|^{-1/2} \textrm{ and }  \\ \epsilon'=\epsilon_{2,1} \delta^{1/2}: \begin{bmatrix}a & 0 \\ 0 & d\end{bmatrix} &\mapsto \chi_2(a)|a|^{1/2} \chi_1(b)|b|^{-1/2}. \end{align*}
In particular, if $\pi$ is irreducible, \cref{theorem.jaquetandkirillov}-(2).(b) implies,
\begin{align*} \pi^{\Kir}&= C^{\sm}_c(\mbb{Q}_\ell^\times, C) + C \cdot 1_{\mbb{Z}_\ell} \cdot \epsilon|_{\mbb{Q}_\ell^\times \times 1} + C \cdot 1_{\mbb{Z}_\ell} \cdot \epsilon'|_{\mbb{Q}_\ell^\times \times 1} \\
&= C^{\sm}_c(\mbb{Q}_\ell^\times, C) + C \cdot 1_{\mbb{Z}_\ell} \cdot \chi_1 \cdot |\cdot|^{1/2} + C \cdot 1_{\mbb{Z}_\ell} \cdot \chi_2 \cdot |\cdot|^{1/2}, \end{align*}
\item Suppose $\pi=\mc{B}(\chi,\chi)$. Then $J(\pi)$ is the non-trivial extension of $\epsilon_{1,2}| \cdot |^{1/2}$ by itself. Thus,
\[ \pi^{\Kir} =C^{\sm}_c(\mbb{Q}_\ell^\times, C)  +  C \cdot 1_{\mbb{Z}_\ell} \cdot \chi \cdot |\cdot|^{1/2} + C \cdot 1_{\mbb{Z}_\ell} \cdot v_\ell \cdot \chi \cdot |\cdot|^{1/2}. \]
\item From (1) and the definition of special representation, it follows that if $\pi$ is the special quotient of $\mc{B}(\chi|\cdot|^{-1}, \chi)$, then $J(\pi)$ is the character 
\[ \begin{bmatrix}a & 0 \\ 0 & d\end{bmatrix} \mapsto \chi(a)|a|^{1/2}\chi(b)|b|^{-3/2}. \]
and 
\[ \pi^\Kir=C^{\sm}_c(\mbb{Q}_\ell^\times, C)  +  C \cdot 1_{\mbb{Z}_\ell} \cdot \chi \cdot |\cdot|^{1/2}.\]
\end{enumerate}   
\end{example}

\subsection{Modular curves}\label{sss.modular-curves}
For $K \leq \GL_2(\mathbb{A}_f)$ a sufficiently small closed compact subgroup (not necessarily open!), we write $Y_K$ for the associated infinite level scheme-theoretic modular curve over $\mathbb{Q}$ parameterizing elliptic curves up to quasi-isogeny with a $K$-orbit of trivializations of their adelic Tate module as in \cite[\S3.1]{Howe.ThepAdicJacquetLanglandsCorrespondenceAndAQuestionOfSerre}.  We write $\omega_{E^\vee}$ for the sheaf on $Y_K$ of invariant differentials on the dual elliptic curve $E^\vee$ with its natural equivariant structure for the normalizer of $K$ (which differs by a twist from the natural equivariant structure on $\omega_E$). As in \cite[\S3.1]{Howe.ThepAdicJacquetLanglandsCorrespondenceAndAQuestionOfSerre}, when $K = \GL_2(\mathbb{Z}_p) \times K^p$ for $K^p \leq \GL_2(\mathbb{A}_f^{(p)})$ a closed compact subgroup, there is a natural integral model of $Y_K$ over $\mathbb{Z}_{(p)}$ parameterizing elliptic curves up to prime-to-$p$ quasi-isogeny with a $K^p$-orbit of trivializations of their adelic Tate modules, and $\omega_{E^\vee}$ extends in the obvious way as an equivariant sheaf for the normalizer of $K^p$ over this integral model. 

\subsection{The Tate curve}\label{ss.tate-curve}
Write $\Tate(q)$ for the Tate curve over $\bbZ((q))$ (see \cite[\S 8.8]{KatzMazur.ArithmeticModuliOfEllipticCurves} for the properties we will use; cf. also \cite[\S 3.2]{Howe.AUnipotentCircleActionOnpAdicModularForms}). We have a canonical identification $\Tate(q)^\wedge = \mbb{G}^\wedge_m$. Using the canonical principal polarization $\Tate(q)^\vee = \Tate(q)$ and transporting the differential $\frac{dt}{t}$ from $\hat{\mathbb{G}}_m$, we obtain a basis $\omega_{\can}$ for $\omega_{\Tate(q)^\vee}$. 

We write $\mathbb{Z}((q^{\mathbb{Q}_{>0}})):=\colim_n \mathbb{Z}((q^{1/n}))$. If we fix a compatible family $\zeta:=(\zeta_n)$ of primitive $n$th roots of unity in $\mathbb{Q}^\ab$, then we obtain a canonical associated trivialization
\[ \varphi_\zeta: \ul{\hat{\mathbb{Z}}^2} \xrightarrow{\sim} T_{\hat{\mathbb{Z}}}\left(\Tate(q)_{\mathbb{Q}^\ab \otimes_{\mathbb{Z}} \mathbb{Z}((q^{\mathbb{Q}_{>0}}))}\right) \]

We note that, if we fix a compatible family $\zeta^{(p)}:=(\zeta_n)_{p \nmid n}$ of prime-to-$p$ power primitive $n$th roots of unity, then we have a trivialization of the prime-to-$p$ Tate module already integrally, 
\[ \varphi_{\zeta^{(p)}}: \ul{\hat{\mathbb{Z}}^2} \xrightarrow{\sim} T_{\hat{\mathbb{Z}}}\left(\Tate(q)_{\mathbb{Z}_{(p)}[\zeta^{(p)}] \otimes_{\mathbb{Z}} \mathbb{Z}((q^{\mathbb{Z}_{(p),>0}}))}\right). \]

\subsection{Modular forms}
For $k \in \mathbb{Z}$, we consider the smooth $\GL_2(\mathbb{A}_f)$-representation $H^0(Y_{\{1\}}, \omega^k_{E^\vee})$ where $\{1\} \leq \GL_2(\mathbb{A}_f)$ is the trivial subgroup. We write $S_k$ for the usual cuspidal subrepresentation, consisting of those sections $s$ such that, for any compatible system of roots unity $\zeta$ as above and for any $g \in \GL_2(\mathbb{A}_f)$, the series obtained by evaluating $s$ on
$(\Tate(q), (\varphi_{\zeta}\otimes \mathbb{Q})g, \omega_\can)$ contains only positive powers of $q$. For any sufficiently small compact open subgroup $K \leq \GL_2(\mathbb{A}_f)$, $S_k^K$ is the space of sections of a line bundle on the smooth compactification of the affine curve $Y_{1}/K=Y_K$, thus $S_k$ is an admissible smooth representation of $\GL_2(\mathbb{A}_f)$. 

For $L/\mathbb{Q}$ any field extension, we write $S_{k,L}=S_{k} \otimes L$; equivalently, one can define $S_{k,L}$ by working everywhere above over $L$.

\subsubsection{Twists}\label{sss.mf-twists} By considering the Weil pairing associated to the unique elliptic curve in the isogeny class with integral Tate module $\hat{\mathbb{Z}}^2$ under the trivialization of the adelic Tate module evaluated on the standard basis vectors $e_1$ and $e_2$, we obtain a map $Y_{\{1\}} \rightarrow \hat{\mathbb{Z}}(1)$ that is $\GL_2(\mathbb{A}_f)$-equivariant when $\GL_2(\mathbb{A}_f)$ acts (on the right) on $\hat{\mathbb{Z}}(1)$ as multiplication by  $\tilde{\det} = \det \cdot |\det|$, where here $|(a_\ell)|=\prod_{\ell} |a_\ell|_\ell$, viewed as an element of $\mathbb{Q}^\times \subseteq \mathbb{A}_f^\times$. 

In particular, for $C/\mathbb{Q}$ algebraically closed, this induces a canonical identification $\pi_0(Y_{\{1\}, C})=\hat{\mathbb{Z}}(1)(C)^\times$. Now, for each smooth character $\kappa: \hat{\mathbb{Z}}^\times \rightarrow C$ there is a unique line $\ell_\kappa \subseteq C^\sm(\hat{\mathbb{Z}}(1)(C)^\times, C)$ consisting of functions $f$ such that $f(xz)=\kappa(z)f(x)$ for all $z \in \hat{\mathbb{Z}}^\times$. Given a subspace $V \subseteq S_k$, we write $V \otimes \kappa \subseteq S_k$ for the subspace $\ell_{\kappa} \cdot V$ (the elements of $\ell_\kappa$ pullback from $\pi_0$ to give locally constant functions on $Y_{\{1\}}$, and we multiply sections by these functions, which does not change the cuspidality condition). If we have fixed a basepoint $\zeta \in \hat{\mathbb{Z}}(1)(C)^\times$ (i.e. a compatible system of roots of unity in $C$), then there is a distinguished basis element $v_{\kappa,\zeta}$ of $\ell_\kappa$ defined by $v_{\kappa,\zeta}(\zeta x)=\kappa(x)$ and, for any $f \in S_{k,C}$, we write $f \otimes \kappa := v_{\kappa,\zeta} f$. We note, in particular, that if $V$ is a subrepresentation of $S_{k,C}$, then $V \otimes \kappa$ is isomorphic as a $\GL_2(\mathbb{A}_f)$ representation to the twist of $V$ by $\kappa \circ \tilde{\det}$.

\subsection{Galois representations associated to modular forms}\label{ss.galois-classical-modular-forms}
Let $C/\mathbb{Q}$ be algebraically closed, and suppose $\pi \in S_{k,C}$ is an irreducible $\GL_2(\mathbb{A}_f)$-subrepresentation. Then it admits a restricted tensor product decomposition $\pi=\bigotimes' \pi_\ell$ where $\pi_\ell$ is an irreducible representation of $\GL_2(\mathbb{Q}_\ell)$ that is spherical for almost all primes $\ell$. 
We recall that such a representation and all of its local components can be defined over any sufficiently large finite extension $L/\mathbb{Q}$. 

\subsubsection{} Given $f \in S_{k,C}^{K_1(N)}$, we obtain a $\mathbb{Q}$-expansion in $qC[[q]]$ by evaluation on 
\[ (\Tate(q), \varphi_\zeta, \omega_\can) \qquad \textit{(see \cref{ss.tate-curve} for this notation)} \]
for any choice of $\zeta$ (the $\varphi_\zeta$ associated to any two choices differ by an element of $K_1(N)$ thus give the same $q$-expansion). With our normalizations, the Hecke operators $T_\ell$, $\ell \nmid n$ and $U_\ell, \ell | N$ described by the usual formulas on $q$-expansions 
are given by $\frac{1}{\ell}$ times the double coset operators for $\begin{bmatrix} \ell & 0 \\ 0 & 1\end{bmatrix}$ --- this can be verified directly by the usual computation with Tate curves, as long as one is careful to use the correct $\omega_{E^\vee}$-equivariant structure rather than the $\omega_E$-equivariant structure (cf. \cite[\S2.2]{Howe.SlopeClassicalityInHigherColemanTheoryViaHighestWeightVectorsInCompletedCohomology} for an essentially equivalent computation). In particular, for $\ell | N$, 
\begin{equation}\label{eq.Uell-double-coset} U_\ell f = \frac{1}{\ell} \sum_{i=0}^{\ell -1} \begin{bmatrix} \ell & i \\ 0 & 1 \end{bmatrix} \cdot f. \end{equation}
The Hecke operator $S_\ell$ is given by $\frac{1}{\ell}$ times the action of $\begin{bmatrix} \ell & 0  \\ 0 & \ell\end{bmatrix}$. 

\newcommand{\WD}{\mathrm{WD}}
\subsubsection{} For $\lambda$ a finite place of $L$, let 
\[ \rho^{\lambda}=\rho^\lambda(\pi): \Gal(\overline{\mbb{Q}}/\mbb{Q}) \rightarrow \GL_2(L_\lambda) \]
be the Galois representation associated to the prime-to-$N$ Hecke eigenvalues of $\pi^{K_1(N)}$ for $N \gg 0$ as in \cite{Deligne.FormesModulairesEtRepresentationslAdiques}; in particular, it is constructed in the cohomology of $H^1(Y, \Sym^k R^1f_* L_{\lambda})$, where $f: E \rightarrow Y$ is the universal elliptic curve, and for $\ell \nmid n$, the geometric Frobenius element at $\ell$ has characteristic polynomial
\[ x^2 - T_\ell x + S_l. \]

\subsubsection{}\label{sss.wd-and-normalization} For each rational prime $\ell$, we write $\rho^{\lambda}_\ell$ for the restriction of $\rho^{\lambda}$ to a decomposition group at $\ell$. We write $\sigma_\ell=\WD(\rho^{\lambda}_\ell)^{F-\mathrm{ss}}$ for the associated Frobenius seimsiple Weil-Deligne representation on a $C$ vector space, which is independent of $\lambda$ by results of Deligne, Langlands \cite{Langlands.ModularFormsAndEllAdicRepresentations}, Carayol \cite{Carayol.SurLesRepresentationslAdiquesAssocieesAuxFormesModulairesDeHilbert}, and Saito \cite{Saito.ModularFormsAndpAdicHodgeTheory} (recall that when $\lambda |\ell$, the Weil-Deligne representation is constructed using the recipe of Fontaine \cite[1.3.5]{Fontaine.RepresentationsEllAdiquesPotentiellementSemiStables}, see also \cite[1.2]{Kisin.OverconvergentModularFormsAndTheFontaineMazurConjecture}). 

Our normalizations are such that, for $\pi_u$ the unitary local Langlands correspondence for $\GL_2(\mathbb{Q}_\ell)$ as in \cite{Deligne.FormesModulairesEtRepresentationsDeGL2},
\[ \pi_\ell = \pi_u(\sigma_\ell)\otimes |\det|_\ell^{-1/2}. \]
This normalization is adapted to comparison with the Kirillov model:
\begin{enumerate}
    \item If $\sigma_\ell$ is a direct sum of two distinct characters, then we have 
\begin{equation}\label{eq.jacquet-irreducible-principal-series} J(\pi_\ell)|_{\mathbb{Q}_\ell^\times \times 1} = \sigma_\ell\end{equation}
under the isomorphism of local class field theory (normalized as in \cref{sss.lcft}). Indeed, \[ \pi_u(\chi_1 \oplus \chi_2) \otimes |\det|_\ell^{-1/2}=\mathcal{B}(\chi_1,\chi_2)\otimes |\det|_\ell^{-1/2}=\mathcal{B}(\chi_1 |\cdot|_\ell^{-1/2}, \chi_2|\cdot|_\ell^{-1/2})\]
so that \cref{eq.jacquet-irreducible-principal-series} follows from \cref{example.kirillov-models}. 

    \item If $\sigma_\ell$ has non-trivial $N$-operator then
\begin{equation}\label{eq.jacquet-special} J(\pi_\ell)|_{\mathbb{Q}_\ell^\times \times 1} = \sigma_\ell^{N=0}. \end{equation}
    Indeed, for $\chi$ the character appearing in $\sigma_\ell^{N=0}$ (so that the action on $\sigma_\ell/\sigma_\ell^{N=0}$ is by $\chi| \cdot |_{\ell}^{-1}$), $\pi_u(\sigma_\ell)$ is the irreducible quotient of $\mathcal{B}(\chi|\cdot|^{-1}, \chi)$, so that $\pi_u(\sigma_\ell)\otimes|\det|_\ell^{-1/2}$ is the irreducible quotient of $\mathcal{B}(\chi|\cdot|_{\ell}^{-3/2}, \chi|\cdot|_\ell^{-1/2})$ and 
    \cref{eq.jacquet-special} follows from \cref{example.kirillov-models}. 
\end{enumerate}

\subsubsection{} Suppose $\kappa: \hat{\mathbb{Z}}^\times \rightarrow L_\lambda$ is a character. Then, for twists as in \cref{sss.mf-twists}, \[ \rho^\lambda(\pi \otimes \kappa) = \rho^\lambda(\pi) \otimes (\kappa \circ \epsilon) \]
 where 
 $\epsilon: \Gal(\overline{\mathbb{Q}}/\mathbb{Q}) \rightarrow \hat{\mathbb{Z}}^\times$
is the cyclotomic character. 

\subsubsection{}\label{sss.ordinary-vector}
Suppose now that $C$ is a complete extension of $L_\lambda$ and let $\rho$ be the extension of $\rho^\lambda$ from $L_\lambda$ to $C$. We say that $\pi$ is \emph{ordinary} if $J(\pi_p)$ contains a character $\epsilon_\ord$ such that $\epsilon_\ord\left(\diag(p,1)\right)$ has absolute value $1$. By the above considerations, this is equivalent to requiring that $\sigma_\ell$ contain a character with $|\sigma_\ell(p)|_p=1$. 

This character is necessarily unique: indeed, if $\sigma_p=\chi_1 \oplus \chi_2$, then $|\chi_1(p)\chi_2(p)|=|p^{k-1}|$ (for example, because the Hodge-Tate weights are $0$ and $k-1$) and $k \geq 2$. When it exists, we write $\chi_\ord = \epsilon_\ord|_{\mathbb{Q}_p^\times \times 1}$, which we may also view as a character of $\Gal(\overline{\mbb{Q}}_p/\mbb{Q}_p)$ appearing in $\sigma_\ell$. 

\begin{remark}\label{remark.canonical-lift-Kirillov-ordinary}
When $\pi$ is ordinary, it follows that, for $U_p$ the double coset operator of \cref{eq.Uell-double-coset}, the canonical lift of the $\epsilon_\ord$-line is a $U_p$-eigenline of unit eigenvalue $\chi_\ord(p)$. This explains the relation with the usual terminology for modular forms. In $\pi_p^\Kir$, this line is spanned by $1_{\mbb{Z}_p} \cdot \chi_\ord$ (cf. \cref{example.kirillov-models}).
\end{remark}

The ordinary character also appears in $\rho_p$:

\begin{lemma}\label{lemma.classical-ordinary}
    If $\pi$ is ordinary, then $\rho_p$ contains a unique invariant line where the Galois action is by the character $\chi_\ord$. On the other hand, if $\rho_p$ is reducible, then $\pi$ is ordinary. 
\end{lemma}
\begin{proof}
Suppose $\pi$ is ordinary. Then, there is a unique twist $\tilde{\pi}$ of $\pi$ such that $\chi_\ord|_{\mathbb{Z}_p^\times}=1$. Since this corresponds to a twist of the Galois representation, we may replace $\pi$ with $\tilde{\pi}$, so it suffices to assume $\chi_\ord|_{\mathbb{Z}_p^\times}=1$. In this case, the canonical lift of $\epsilon_\ord$ is a new vector with unit $U_p$-eigenvalue $\chi_\ord(p)$. We can thus apply \cite[Theorem 2]{Wiles.OnOrdinaryLambdaAdicRepresentationsAssociatedToModularForms} (note that our $\rho_p$ is the dual of the representation considered in loc cit., which can be seen as the eigenvalues of arithmetic and geometric Frobenius elements are swapped) to see that $\rho_p$ is reducible with an unramified line where geometric Frobenius acts by $\chi_\ord(p)$. We conclude because this line is unique since the Hodge-Tate weights are $0$ and $k-1$. 

Conversely, suppose $\rho_p$ is reducible and let $\ell$ be a one-dimensional subrepresentation with quotient $q$. The Hodge-Tate weight of $\ell$ must be $0$ or $k-1$. If it is zero then $\ell$ is unramified and $\ell=\WD(\ell)\subseteq \WD(\rho_p)$ is an ordinary line. If it is $k-1$, then, we note that  $\WD(\rho_p)$ is an extension of  $q=\WD(q)$ (which is potentially unramified since it is of Hodge-Tate weight zero) by $\WD(\ell)$. Writing $\chi$ for the character on $\WD(\ell)$, we find $|\chi(p)|=|p^{k-1}|$, so there is no non-trivial extension of $\WD(q)$ by $\WD(\ell)$. Thus $\WD(q)$ is contained in $\WD(\rho_p)$ and gives an ordinary line.  
\end{proof}

\begin{remark}
    As we will recall in \cref{ss.Galois-reps-ordinary-forms}, there is a generalization of \cref{lemma.classical-ordinary} to Hida-ordinary $p$-adic modular forms of arbitrary weight due to Kisin \cite[6.13]{Kisin.OverconvergentModularFormsAndTheFontaineMazurConjecture}.  
\end{remark}

\section{Representation theory of $\mathbb{Z}_p(1)$, $\mu_{p^\infty}$, and $\tilde{\mu}_{p^\infty}$}\label{s.rep-theory}

In this section we discuss the representation theory of $\mathbb{Z}_p(1)$, $\mu_{p^\infty}$, and $\tilde{\mu}_{p^\infty}$ using integral $p$-adic Fourier theory. This is a very special case compared to other $p$-divisible groups, as the dual groups  $\mathbb{Q}_p/\mathbb{Z}_p$, $\mathbb{Z}_p$, and $\mathbb{Q}_p$ are all locally profinite, so that we can decompose such representations explicitly by viewing them as sheaves on locally profinite sets. Using this interpretation, we define topological coinvariants for $\mu_{p^\infty}$ and $\tilde{\mu}_{p^\infty}$-representations, and relate these two notions (\cref{lemma.topological-coinvariants-agree}). 

In later sections, we will apply this study to functions on ordinary Igusa varieties. One interesting consequence of the general study is that the density of ``finite level" (Mantovan) functions in the ``infinite level" (Caraiani-Scholze)  functions is in fact a general feature of representations of $\mathbb{Z}_p(1)$ (see \cref{cor.density-zp1}). 

\subsection{Topological $\mathbb{Z}_p$-modules and Hopf algebras}
By a topological $\mathbb{Z}_p$-module, we mean a linearly topologized $\mathbb{Z}_p$-module $M$ whose topology is coarser than the $p$-adic topology (i.e., if $U$ is an open neighborhood of the identity, then $p^n M \subseteq U$ for $n \gg 0$). We say $M$ is complete if $M = \varprojlim M/U$, where $U$ runs over the open sub-modules. We will consider the completed tensor product $M \hat{\otimes}_{\mathbb{Z}_p} N$ of topological $\mathbb{Z}_p$-modules as in \cite[\href{https://stacks.math.columbia.edu/tag/0AMU}{Tag 0AMU}]{stacks-project}. 

A complete topological Hopf algebra over $\mathbb{Z}_p$ is a complete topological $\mathbb{Z}_p$-algebra $A$ equipped with a continuous counit $A \rightarrow \mathbb{Z}_p$ and comultiplication $A \rightarrow A \hat{\otimes}_{\mathbb{Z}_p} A$  satisfying the usual identities and admitting a coinverse.

\subsection{Some $p$-adic Fourier theory}

We consider the $p$-divisible group $\mu_{p^\infty}=\varinjlim_n \mu_{p^n}$ over $\Spf \mathbb{Z}_p$ (viewed as a functor on $\Nilp_{\mathbb{Z}_p}$, the category of $\mathbb{Z}_p$-algebras where $p$ is nilpotent). We write $\mathbb{Z}_p(1)=T_p \mu_{p^\infty}=\varprojlim_n \mu_{p^n}$ for its Tate module. The Cartier dual of the finite flat group scheme $\mu_{p^n}$ is $(\mathbb{Q}_p/\mathbb{Z}_p)[p^n]=\frac{1}{p^n}\mathbb{Z}_p/\mathbb{Z}_p$, and the explicit description of Cartier duality in terms of Hopf algebras then implies the following integral $p$-adic Fourier theory (a special case of \cite[Theorem 7.0.1]{GrahamVanHoftenHowe.FourierTheory}). In the statement, for $M$ a topological $\mathbb{Z}_p$-module, we write $M^*=\Hom_{C^0,\mathbb{Z}_p}(M, \mathbb{Z}_p)$ for the space of continuous functionals on $M$, equipped with the weak topology.  

\begin{proposition}\label{prop.p-adic-fourier-theory}
Integration over the universal characters 
\[ \mu_{p^\infty} \times \mathbb{Z}_p \rightarrow \mu_{p^\infty} \textrm{ and } \mathbb{Z}_p(1) \times \mathbb{Q}_p/\mathbb{Z}_p \rightarrow \mu_{p^\infty}\]
induces isomorphisms of complete topological Hopf algebras over $\mathbb{Z}_p$:
\[ \mathcal{O}(\mu_{p^\infty})^*=\mathcal{O}(\mathbb{Z}_p)=C(\mathbb{Z}_p, \mathbb{Z}_p) \textrm{ and } 
   \mathcal{O}(\mathbb{Z}_p(1))^*=\mathcal{O}(\mathbb{Q}_p/\mathbb{Z}_p)=C(\mathbb{Q}_p/\mathbb{Z}_p, \mathbb{Z}_p). \]
\end{proposition}

\begin{remark}
    The first isomorphism in \cref{prop.p-adic-fourier-theory} can also be seen using the M\"ahler basis as in \cite[Examples 7.0.3-4]{GrahamVanHoftenHowe.FourierTheory}; note here that the topology on $\mathcal{O}(\mu_{p^\infty})$ is not $p$-adic! By \cite[Proposition 7.0.2]{GrahamVanHoftenHowe.FourierTheory}, this isomorphism also exchanges multiplication by the coordinate function $z \in C(\mathbb{Z}_p, \mathbb{Z}_p)$ with the dual of the action of the invariant derivation $t\partial_t$ on $\mathcal{O}(\mu_{p^\infty})$.  
\end{remark}

We will also consider the universal cover 
\[ \tilde{\mu}_{p^\infty} = \varprojlim_n \mu_{p^\infty} \]
where the transition maps are multiplication by $p$. We have a natural inclusion 
\[ \mathbb{Z}_p(1)=\varprojlim_{n} \mu_{p^n} \hookrightarrow \tilde{\mu}_{p^\infty} \]
that extends to an exact sequence of fpqc-sheaves of abelian groups on $\Nilp_{\mathbb{Z}_p}$, 
\begin{equation}\label{eq.ses} 1 \rightarrow \mathbb{Z}_p(1) \rightarrow \tilde{\mu}_{p^\infty} \rightarrow \mu_{p^\infty} \rightarrow 1 \end{equation}
by projecting from $\tilde{\mu}_{p^\infty}$ onto the first component in the limit.

We note the inclusion $\mathbb{Z}_p(1) \rightarrow \tilde{\mu}_{p^\infty}$ also induces an isomorphism
\[ \mathbb{Z}_p(1)[1/p] = \tilde{\mu}_{p^\infty}. \]
Since $\mathbb{Z}_p(1)[1/p]=\varinjlim_n \frac{1}{p^n}\mathbb{Z}_p(1)$, we can use this to compute (by taking maps to $\mathbb{A}^1$ and applying the universal property of colimits)
\[ \mathcal{O}(\tilde{\mu}_{p^\infty})= \varprojlim_n \mathcal{O}\left(\frac{1}{p^n}\mathbb{Z}_p(1)\right). \]
We equip $\mathcal{O}(\tilde{\mu}_{p^\infty})$ with the inverse limit topology so that it is naturally a complete topological Hopf algebra over $\mathbb{Z}_p$.  

\begin{proposition}\label{prop.universal-cover-duals}
There is a natural isomorphism 
\[ \mathcal{O}(\tilde{\mu}_{p^\infty})^* = C(\mathbb{Q}_p, \mathbb{Z}_p) \]
\end{proposition}
\begin{proof}
    Since the topology on $\mathcal{O}(\mathbb{Z}_p(1))$ is $p$-adic, we find
\begin{align*} \mathcal{O}(\tilde{\mu}_{p^\infty})^* &= \varprojlim_k  \Hom_{\mathbb{Z}_p}(\varinjlim_n \mathcal{O}(1/p^n \mathbb{Z}_p(1)), \mathbb{Z}/p^k\mathbb{Z})\\
&=\varprojlim_k \varprojlim_n 
C\left(\frac{1}{p^n}\mathbb{Z}_p, \mathbb{Z}/p^k\right)\\
&= C(\mathbb{Q}_p, \mathbb{Z}_p).\end{align*}
\end{proof}

\begin{remark} Integral $p$-adic Fourier theory for universal covers of more general $p$-divisible groups, and its relation to an analytic Fourier theory for Banach-Colmez spaces, will be discussed in greater generality in \cite{GrahamVanHoftenHowe.TowardsAFourierTheoryForBanachColmezSpaces}.
\end{remark}

\subsection{Representations of $\mathbb{Z}_p(1)$}

\begin{definition}\hfill
\begin{itemize}
\item A $\mathbb{Z}_p(1)$-representation is a complete topological $\mathbb{Z}_p$-module $M$ with a co-action 
    \[ M \rightarrow M \hat{\otimes}_{\mathbb{Z}_p
    } \mathcal{O}(\mathbb{Z}_p(1)). \]
\item Given a $\mathbb{Z}_p(1)$-representation $M$, the invariants $M^{\mathbb{Z}_p(1)}$ are the preimage of $M = M \hat{\otimes} \mathbb{Z}_p \subseteq M \hat{\otimes}_{\mathbb{Z}_p
    } \mathcal{O}(\mathbb{Z}_p(1)) $ under the co-action map. 
\end{itemize}
\end{definition}

Given a $\mathbb{Z}_p(1)$-representation $M$, we obtain a $C(\mathbb{Q}_p/\mathbb{Z}_p, \mathbb{Z}_p)=\mathcal{O}(\mathbb{Z}_p(1))^*$-action on $M$ by contracting with the co-action. The following lemma uses the discreteness of $\mathbb{Q}_p/\mathbb{Z}_p$ to decompose  
$\mathbb{Z}_p(1)$-representations into isotypic subspaces. For $u \in \mathbb{Q}_p/\mathbb{Z}_p$, we write    $1_{u}\in C(\mathbb{Q}_p/\mathbb{Z}_p, \mathbb{Z}_p)$ for the indicator function of $u$.

\begin{proposition}
    Let $M$ be a $\mathbb{Z}_p(1)$-representation. Then, for $u \in \mathbb{Q}_p/\mathbb{Z}_p$, $\mathbb{Z}_p(1)$ acts by the character $u$ on $1_u \cdot M$, and 
    \[ M = \left(\oplus_{u \in \mathbb{Q}_p/\mathbb{Z}_p} 1_{u} \cdot M\right)^\wedge \]

\end{proposition}
\begin{proof}
The description of the action follows since $1_u \cdot m$, viewed as a sheaf on the discrete set $\mathbb{Q}_p/\mathbb{Z}_p$, is supported at the character $u$. 

Since the $1_u$'s are independent idempotents, to obtain the description as a completed direct sum, it suffices to show that for any $m \in M$, 
\[ m = \sum_{u \in \mathbb{Q}_p/\mathbb{Z}_p} 1_u \cdot m. \]
To see this, it suffices to verify the identity modulo any open submodule $U$. 

Fix such a $U$, and let $\overline{m}$ denote the image of $m$ in $M/{U}$. Recall that such a $U$ contains $p^kM$ for some $k$. We note that the image of $\overline{m}$ under the coaction is contained in $M/U \otimes_{\mathbb{Z}/p^k\mathbb{Z}} \mathcal{O}(\mu_{p^n, \mathbb{Z}/p^k\mathbb{Z}})$ for some $n$, since $\mathcal{O}(\mathbb{Z}_p(1)_{\mathbb{Z}/p^k\mathbb{Z}})=\colim_n \mathcal{O}((\mu_{p^n})_{\mathbb{Z}/p^k\mathbb{Z}})).$
We have 
\[ \mathcal{O}((\mu_{p^n})_{{\mathbb{Z}/p^k\mathbb{Z}}})^*= C\left(\frac{1}{p^n}\mathbb{Z}/\mathbb{Z}, \mathbb{Z}/p^k\mathbb{Z}\right). \]
Since the action of $1_{u}$ for $u \not\in \frac{1}{p^n}\mathbb{Z}/\mathbb{Z}$ on $\overline{m}$ is by zero and the unit in $C(\frac{1}{p^n}\mathbb{Z}/\mathbb{Z}, \mathbb{Z}/p^k\mathbb{Z})$ is $\sum_{u \in \frac{1}{p^n}\mathbb{Z}/\mathbb{Z}} 1_u$, we conclude. 
\end{proof}

In the following, for $n \geq 0$ we write $1_{p^{-n}\mathbb{Z}_p}$ for the indicator function of $p^{-n}\mathbb{Z}_p/\mathbb{Z}_p$ in $C(\mathbb{Q}_p/\mathbb{Z}_p, \mathbb{Z}_p)$. 
\begin{corollary}\label{cor.density-zp1}
If $M$ is a $\mathbb{Z}_p(1)$-representation, then multiplication by $1_{p^{-n}\mathbb{Z}_p}$ is a projection onto the invariants $M^{p^n\mathbb{Z}_p(1)}$, and the smooth vectors 
\[ M^\sm = \cup_n M^{p^n \mathbb{Z}_p(1)} = \cup_n 1_{p^{-n} \mathbb{Z}_p(1)} \cdot M\]
are dense in $M$. 
\end{corollary}

\subsection{$\mu_{p^\infty}-$ representations}\label{ss.mupinf-reps}
\begin{definition}
A $\mu_{p^\infty}$-representation is a complete topological $\mathbb{Z}_p$-module $M$ with a co-action of $\mathcal{O}(\mu_{p^\infty})$, 
\[ M \rightarrow M \hat{\otimes}_{\mathbb{Z}_p
    } \mathcal{O}(\mu_{p^\infty}). \]
\end{definition}

Given a $\mu_{p^\infty}$ representation $M$, we obtain a $C(\mathbb{Z}_p, \mathbb{Z}_p)=\mathcal{O}(\mu_{p^\infty})^*$-action on $M$ by contracting with the co-action. 

\begin{definition}\label{def.top-coinv-mupinfty}
    For $M$ a $\mu_{p^\infty}$-representation, the topological co-invariants of $M$ are 
    \[ M_{\mu_{p^\infty}}:= M \hat{\otimes}_{C(\mathbb{Z}_p,\mathbb{Z}_p)} \mathbb{Z}_p, \] 
    where $\mathbb{Z}_p$ is the fiber at $0$ as a $C(\mathbb{Z}_p,\mathbb{Z}_p)$-module, i.e. $f \cdot z = f(0)z$. Equivalently, 
\begin{equation} \label{eq.limit-pres-mupinfty-coinvariants} M_{\mu_{p^\infty}}= \varprojlim_{U,k} \left( M/U \otimes_{C(\mathbb{Z}_p, \mathbb{Z}_p)} \mathbb{Z}/p^k\mathbb{Z}\right) \end{equation}
where $U$ runs over all open submodules of $M$ and $k$ runs over non-negative integers. 
\end{definition}

We note that in \cref{eq.limit-pres-mupinfty-coinvariants}, it suffices to take only the limit over terms corresponding to a single $k_U$ such that $p^{k_U}M \in U$, in which case it equals
\begin{equation} \label{eq.limit-pres-mupinfty-coinvariants-v2} M_{\mu_{p^\infty}}= \varprojlim_{U} \left( M/U \otimes_{C(\mathbb{Z}_p, \mathbb{Z}/p^{k_U} \mathbb{Z})} \mathbb{Z}/p^{k_U}\mathbb{Z}\right) \end{equation}

\subsection{$\tilde{\mu}_{p^\infty}$-representations}\label{ss.mupinftilde-reps}

\begin{definition}
    A $\tilde{\mu}_{p^\infty}$-representation is complete topological $\mathbb{Z}_p$-module $M$ with a co-action 
    $M \rightarrow M \hat{\otimes}_{\mathbb{Z}_p} \mathcal{O}(\tilde{\mu}_{p^\infty}).$ 
\end{definition}

Given a $\tilde{\mu}_{p^\infty}$-representation $M$, we obtain a 
$C(\mathbb{Q}_p, \mathbb{Z}_p)$-action on $M$ by contraction with the co-action. 

\begin{lemma}\label{lem.zp1invariants-rep}
    Let $M$ be a $\tilde{\mu}_{p^\infty}$-representation. Then the invariants for the induced $\mathbb{Z}_p(1)$-representation, $M^{ \mathbb{Z}_p(1)}$, are naturally a $\mu_{p^\infty}$-representation.  
\end{lemma}
\begin{proof}

The action of $\mathbb{Z}_p(1)$ on $M$ is given by the composition
\[ M \rightarrow M \hat{\otimes}_{\mbb{Z}_p} \mathcal{O}(\tilde{\mu}_{p^\infty}) \rightarrow M \hat{\otimes} \mathcal{O}(\mathbb{Z}_p(1)) \]
induced by the restriction map $\mathcal{O}(\tilde{\mu}_{p^\infty}) \rightarrow \mathcal{O}(\mathbb{Z}_p(1))$. 
Thus $M^{\mathbb{Z}_p(1)}$ can be identified with the preimage of $M \hat{\otimes}_{\mbb{Z}_p} (\mathcal{O}(\tilde{\mu}_{p^\infty})^{\mathbb{Z}_p(1)})$. We claim that \cref{eq.ses} implies \[ \mathcal{O}(\mu_{p^\infty})=\mathcal{O}(\tilde{\mu}_{p^\infty})^{\mathbb{Z}_p(1)}, \]
which will give rise to the desired action of $\mu_{p^\infty}$. To see this claim, note that \cref{eq.ses} implies there is a co-equalizer diagram of fpqc sheaves of sets
\[\begin{tikzcd}
	{\mathbb{Z}_p(1)\times \tilde{\mu}_{p^\infty}} && {\tilde{\mu}_{p^\infty}} & {\mu_{p^\infty}.}
	\arrow["{(u_1,u_2)\mapsto u_2}"', shift right=2, from=1-1, to=1-3]
	\arrow["{(u_1,u_2)\mapsto u_1u_2}", shift left=2, from=1-1, to=1-3]
	\arrow[from=1-3, to=1-4]
\end{tikzcd}\]
The claim then follows by taking maps to $\mathbb{A}^1_{\mathbb{Z}_p}$. 
\end{proof}

\begin{definition}\label{def.top-coinv-tildemupinfty}
    For $M$ a $\tilde{\mu}_{p^\infty}$-representation, the topological co-invariants of $M$ are 
    \[ M_{\tilde{\mu}_{p^\infty}}:= M \hat{\otimes}_{C(\mathbb{Q}_p,\mathbb{Z}_p)} \mathbb{Z}_p, \] 
    where $\mathbb{Z}_p$ is the fiber at $0$ as a $C(\mathbb{Q}_p,\mathbb{Z}_p)$-module, i.e. $f \cdot z = f(0)z$. Equivalently, we have
\begin{equation}\label{eq.limit-pres-tildemupinfty-coinvariants} M_{\tilde{\mu}_{p^\infty}}= \varprojlim_{U,k} \left( M/U \otimes_{C(\mathbb{Z}_p, \mathbb{Z}_p)} \mathbb{Z}/p^k\mathbb{Z}\right) \end{equation}
where $U$ runs over all open submodules of $M$ and $k$ runs over all non-negative integers. 
\end{definition}

We note that in \cref{eq.limit-pres-tildemupinfty-coinvariants}, it suffices to take only the limit over terms corresponding to a single $k_U$ such that $p^{k_U}M \in U$, in which case it equals
\begin{equation} \label{eq.limit-pres-tildemupinfty-coinvariants-v2} M_{\tilde{\mu}_{p^\infty}}= \varprojlim_{U} \left( M/U \otimes_{C(\mathbb{Z}_p, \mathbb{Z}/p^{k_U} \mathbb{Z})} \mathbb{Z}/p^{k_U}\mathbb{Z}\right) \end{equation}

We may view a $\tilde{\mu}_{p^\infty}$-representation, via the $C(\mathbb{Q}_p, \mathbb{Z}_p)$-action, as a  sheaf on $\mathbb{Q}_p$. The following result says essentially that, when taking the fiber at $0$, we may first restrict to a sheaf on $\mathbb{Z}_p$ and then take the fiber at $0$. 

\begin{lemma}\label{lemma.topological-coinvariants-agree}
    If $M$ is a $\tilde{\mu}_{p^\infty}$-representation, then the natural map
\[ M^{\mathbb{Z}_p(1)} \rightarrow M_{\tilde{\mu}_{p^\infty}} \]
    induces an isomorphism
    \[ \left( M^{\mathbb{Z}_p(1)} \right)_{\mu_{p^\infty}} \xrightarrow{\sim} M_{\tilde{\mu}_{p^\infty}}. \]
    Moreover, $m$ and $1_{\mathbb{Z}_p} \cdot m$ have the same image for any $m \in M$. 
\end{lemma}
\begin{proof}
The statement ``moreover" is trivial since $1_{\mathbb{Q}_p} - 1_{\mathbb{Z}_p}$ takes value $0$ at $0 \in \mathbb{Z}_p$.  

For the first statement, we first note that 
\[ M \hat{\otimes}_{C(\mathbb{Q}_p, \mathbb{Z}_p)} \mathbb{Z}_p = \left(M \hat{\otimes}_{C(\mathbb{Q}_p, \mathbb{Z}_p)} C(\mathbb{Z}_p,\mathbb{Z}_p)\right) \hat{\otimes}_{\mathbb{C}(\mathbb{Z}_p,\mathbb{Z}_p)}\mathbb{Z}_p.\]
The result then follows since $1_{\mathbb{Z}_p}$ is an idempotent with $1_{\mathbb{Z}_p}\cdot C(\mathbb{Q}_p,\mathbb{Z}_p)=C(\mathbb{Z}_p,\mathbb{Z}_p)$ and $1_{\mathbb{Z}_p} \cdot M=M^{\mathbb{Z}_p(1)}$ (by \cref{cor.density-zp1}), so that 
\[ M \hat{\otimes}_{C(\mathbb{Q}_p, \mathbb{Z}_p)} C(\mathbb{Z}_p,\mathbb{Z}_p)=\left(1_{\mathbb{Z}_p} \cdot M \right) \hat{\otimes}_{1_{\mathbb{Z}_p}\cdot C(\mathbb{Q}_p, \mathbb{Z}_p)} \left(1_{\mathbb{Z}_p} \cdot C(\mathbb{Q}_p,\mathbb{Z}_p)\right)  = M^{\mathbb{Z}_p(1)} \]
as a $C(\mathbb{Z}_p,\mathbb{Z}_p)$-module. 
\end{proof}

\subsection{Variant}
The results of this section hold when the base $\Spf \mathbb{Z}_p$ is replaced with $\Spf R$ for any $p$-adically complete ring $R$ (with its $p$-adic topology).   

\section{Ordinary Igusa varieties for $\GL_2$ and their functions}\label{s.ordinary-igusa-varieties}

In this section we define the main spaces arising in the theory and explain how they relate to one another. 

\subsection{Definitions}

Let $b=\begin{bmatrix}p^{-1} & 0 \\ 0 & 1 \end{bmatrix}$, an element of $\GL_2(\Qpbreve)$ representing the rank two isocrystal given by the covariant Dieudonn\'{e} module of the $p$-divisible group $\mu_{p^\infty} \times \mathbb{Q}_p/\mathbb{Z}_p$ over $\mathbb{F}_p$ (normalized as in \cite{ScholzeWeinstein.ModuliOfpDivisibleGroups}). Associated to $b$ there is a Caraiani-Scholze Igusa formal scheme $\Ig^b_{\Zpbreve}$ (a special case of the construction denoted $\mathrm{Ig}^b_{\mathcal{O}_{\breve{E}}}$ in \cite[bottom of p. 718]{CaraianiScholze.OnTheGenericPartOfTheCohomologyOfCompactUnitaryShimuraVarieties}). It admits a natural model $\Ig^b$ over $\mathbb{Z}_p$. We now make its definition explicit and define some related moduli problems (see \cite[Section 4]{Howe.AUnipotentCircleActionOnpAdicModularForms} for more details, noting that what we call $\Ig^b$ here is written as $\Ig_\CS$ in \cite{Howe.AUnipotentCircleActionOnpAdicModularForms}).

\subsubsection{}For $R$ a $p$-adically complete ring, we let $\Nilp_{R}$ denote the category of $R$-algebras where $p$ is nilpotent. For $G$ a $p$-divisible group over $R$, viewed as an fpqc sheaf on $\Nilp_{R},$ we write $G^\circ$ for the formal neighborhood of the identity section $e_G: \Spec R \rightarrow G$, and let $G^\et:=G/G^\circ$ (note that neither $G^\circ$ nor $G^\et$ is necessarily a $p$-divisible group). 
\begin{definition}
We consider the following moduli problems on $\Nilp_{\mathbb{Z}_p}$. 

\begin{itemize}
\item $\Ig^b(R)=\{ (E, \varphi_p, \varphi^{(p)}) \} / \sim$ where: 
\begin{itemize}
    \item $E/R$ is an elliptic curve, 
    \item $\varphi_p: E[p^\infty] \rightarrow \mu_{p^\infty} \times \mathbb{Q}_p/\mathbb{Z}_p$ is a quasi-isogeny,  
    \item $\varphi^{(p)}: V_{\mathbb{A}_f^{(p)}} E\xrightarrow{\sim} \ul{\mathbb{A}_f^{(p)}}^2$ is an isomorphism,
\end{itemize} 
and the equivalence relation is through quasi-isogeny of $E$. 
\item $\Ig_{\Mant}(R)=\{ E, \varphi^\circ_p, \varphi^{\et}_p, \varphi^{(p)}) \}/\sim$ where
\begin{itemize}
    \item $E/R$ is an elliptic curve, 
    \item $\varphi^\circ_p: E[p^\infty]^\circ \rightarrow \mu_{p^\infty}$ and $\varphi^\et_p: E[p^\infty]^\et \rightarrow \mathbb{Q}_p/\mathbb{Z}_p$ are isomorphisms, 
    \item  $\varphi^{(p)}: V_{\mathbb{A}_f^{(p)}} E\xrightarrow{\sim} \ul{\mathbb{A}_f^{(p)}}^2$ is an isomorphism, 
\end{itemize}
and the equivalence relation is through prime-to-$p$ quasi-isogeny of $E$. 
\item $\Ig_\Katz(R)=\{ (E, \varphi^\circ_p, \varphi^{(p)}) \} / \sim$ where 
\begin{itemize}
    \item $E/R$ is an elliptic curve, 
    \item $\varphi^\circ_p: E[p^\infty]^\circ \rightarrow \mu_{p^\infty}$ is an isomorphism,
    \item $\varphi^{(p)}: V_{\mathbb{A}_f^{(p)}} E\xrightarrow{\sim} \ul{\mathbb{A}_f^{(p)}}^2$ is an isomorphism,   
\end{itemize}
and the equivalence relation is through prime-to-$p$ quasi-isogeny of $E$.
\item $\Ig_{\Katz'}(R)=\{ (E, \varphi^{\et}_p, \varphi^{(p)}) \} / \sim$ where
\begin{itemize}
    \item $E/R$ is an elliptic curve, 
    \item $\varphi_p^\et: E[p^\infty]^{\et} \rightarrow \mathbb{Q}_p/\mathbb{Z}_p$ is an isomorphism,  
    \item $\varphi^{(p)}: V_{\mathbb{A}_f^{(p)}} E\xrightarrow{\sim} \ul{\mathbb{A}_f^{(p)}}^2$ is an isomorphism, and the equivalence relation is through prime-to-$p$ quasi-isogeny of $E$. 
\end{itemize}
\end{itemize}
\end{definition}

The moduli problem $\Ig^b$ is equivalent to analogous moduli problem where $\varphi_p$ is required to be an isomorphism and the equivalence relation is instead prime-to-$p$ quasi-isogeny. It follows that there are natural forgetful maps from $\Ig^b \rightarrow \Ig_\bullet$ for $\bullet=\Mant, \Katz, \Katz'$. 

\subsubsection{}We write $\tilde{G}_b= \begin{bmatrix} \mathbb{Q}_p^\times & \tilde{\mu}_{p^\infty} \\ 0 & \mathbb{Q}_p^\times \end{bmatrix}$ for the quasi-isogeny group of $\mu_p^\infty \times \mathbb{Q}_p/\mathbb{Z}_p$, where we note $\tilde{\mu}_{p^\infty}=\Hom(\mathbb{Q}_p, \mu_{p^\infty})$. We write $\tilde{K}_b=\begin{bmatrix} \mathbb{Z}_p^\times & \mathbb{Z}_p(1) \\ 0 & \mathbb{Z}_p^\times \end{bmatrix}$ for the automorphism group of $\mu_{p^\infty} \times \mathbb{Q}_p/\mathbb{Z}_p$, which is naturally a  subgroup of $\tilde{G}_b$. We write $T$ for the diagonal torus in $\GL_2$, so that $\tilde{G}_b=T(\mathbb{Q}_p) \ltimes \tilde{\mu}_{p^\infty}$ and $\tilde{K}_b=T(\mathbb{Z}_p) \ltimes \mu_{p^\infty}$. 

There is a natural action of $\tilde{G}_b$ on $\Ig^b$ by composition with $\varphi_p$. For this action, the natural forgetful maps induce identifications of fpqc quotients
\begin{equation}\label{eq.quotient-presentation} \Ig_\Mant = \Ig^b / \begin{bmatrix} 1 & \mathbb{Z}_p(1) \\ 0 & 1 \end{bmatrix}, \; \Ig_\Katz = \Ig^b / \begin{bmatrix} 1 & \mathbb{Z}_p(1) \\ 0 & \mathbb{Z}_p^\times \end{bmatrix}, \; \textrm{ and } \Ig_{\Katz'} = \Ig^b / \begin{bmatrix} \mathbb{Z}_p^\times & \mathbb{Z}_p(1) \\ 0 & 1 \end{bmatrix}. \end{equation}

\begin{remark}
The space $\Ig^b$ and will be the main object of study. The spaces $\Ig_\Katz$ and $\Ig_{\Katz'}$ are canonically isomorphic by using the Weil pairing on $E[p^\infty]$, but this isomorphism is only equivariant up to a twist (cf. \cref{ss.vmant-from-vkatz}). Although $\Ig_\Katz$ is the traditional home of the theory of $p$-adic modular forms and is the moduli problem used to formulate this theory in \cite{Howe.AUnipotentCircleActionOnpAdicModularForms}, to obtain the usual normalization of the $U_p$ operator as in \cite{Howe.SlopeClassicalityInHigherColemanTheoryViaHighestWeightVectorsInCompletedCohomology} and a better compatibility with the theory of the smooth Kirillov model as in \cref{sss.wd-and-normalization}, we will work everywhere with $\Ig_{\Katz'}$ below.   
\end{remark}

As recalled in \cite{Howe.AUnipotentCircleActionOnpAdicModularForms}, each of these moduli problems is represented by an affine $p$-adic formal scheme. We write $\mathbb{V}_b=\mathcal{O}(\Ig^b)$, $\mathbb{V}_{\Mant}=\mathcal{O}(\Ig_\Mant)$, $\mathbb{V}_{\Katz}=\mathcal{O}(\Ig_{\Katz})$ and  $\mathbb{V}_{\Katz'}=\mathcal{O}(\Ig_{\Katz'})$. From \cref{eq.quotient-presentation}, we find (arguing as in the proof of \cref{lem.zp1invariants-rep}), 

\begin{equation}\label{eq.functions-as-invariants} \mbb{V}_\Mant = \mbb{V}_b^{\begin{bmatrix} 1 & \mathbb{Z}_p(1) \\ 0 & 1 \end{bmatrix}}, \; \mbb{V}_\Katz = \mathbb{V}_b^{\begin{bmatrix} 1 & \mathbb{Z}_p(1) \\ 0 & \mathbb{Z}_p^\times \end{bmatrix}}, \; \textrm{ and } \mathbb{V}_{\Katz'} = \mathbb{V}_{b}^{ \begin{bmatrix} \mathbb{Z}_p^\times & \mathbb{Z}_p(1) \\ 0 & 1 \end{bmatrix}}. \end{equation}

\subsection{The action of continuous functions}\label{ss.action-of-cont-functions}

The action of $\tilde{\mu}_{p^\infty}$ on $\Ig^b$ equips $\mathbb{V}_b$ with the structure of a $\tilde{\mu}_{p^\infty}$-representation in the sense of \cref{ss.mupinftilde-reps}. As in \cref{ss.mupinftilde-reps}, we thus obtain an action of $C(\mathbb{Q}_p, \mathbb{Z}_p)$ on $\mathbb{V}_b$. By construction, it is compatible with the actions of $T(\mathbb{Q}_p) \times \GL_2(\mathbb{A}_f^{(p)})$ in that it extends to an action of 
\[\left(  T(\mathbb{Q}_p) \ltimes C(\mathbb{Q}_p, \mathbb{Z}_p) \right) \times \GL_2(\mathbb{A}_f^{(p)}). \]

In light of \cref{eq.functions-as-invariants}, \cref{lem.zp1invariants-rep} then yields a $\mu_{p^\infty}$-representation on $\mathbb{V}_\Mant$. One can also see this geometrically as in \cite[\S5]{Howe.AUnipotentCircleActionOnpAdicModularForms}: from the quotient presentation of \cref{eq.quotient-presentation}, we obtain an action of $\mu_{p^\infty}=\tilde{\mu}_{p^\infty}/\mathbb{Z}_p(1)$ on $\Ig_\Mant$, and thus a $\mu_{p^\infty}$-representation on  $\Ig_\Mant$. From either perspective, one then obtains as in \cref{ss.mupinf-reps} a $C(\mathbb{Z}_p, \mathbb{Z}_p)$ action on $\mbb{V}_\Mant$; this is the action of \cite[\S7]{Howe.AUnipotentCircleActionOnpAdicModularForms}.

We note that, by \cref{cor.density-zp1}, for $1_{\mathbb{Z}_p} \in C(\mathbb{Q}_p, \mathbb{Z}_p)$ the indicator function of $\mathbb{Z}_p \subseteq \mathbb{Q}_p$, 
\begin{equation}\label{eq.V-mant-proj} \mathbb{V}_\Mant=1_{\mathbb{Z}_p} \cdot \mathbb{V}_b. \end{equation}

\subsection{Reconstruction of $\mathbb{V}_b$ from $\mathbb{V}_\Mant$}\label{ss.vb-from-vmant} The module $\mathbb{V}_b$ with its $C(\mathbb{Q}_p, \mathbb{Z}_p)$ action can be reconstructed from $\mathbb{V}_\Mant$ with its $C(\mathbb{Z}_p, \mathbb{Z}_p)$ action as follows: From \cref{eq.functions-as-invariants}, we see that the right action of $\begin{bmatrix} p^{-1} & 0 \\ 0 & 1 \end{bmatrix}$ induces an isomorphism  
\[ \Ig_\Mant=\Ig_b/\mathbb{Z}_p(1) \xrightarrow{\sim} \Ig_b/p\mathbb{Z}_p(1) \]
and, by composition with the projection
\[ \Ig_b/p\mathbb{Z}_p(1) \rightarrow \Ig_b/\mathbb{Z}_p(1)=\Ig_\Mant, \]
a map $F: \Ig_\Mant \rightarrow \Ig_\Mant$.
It is a Frobenius lift (cf. \cite[Remark 5.2.1]{Howe.AUnipotentCircleActionOnpAdicModularForms}).

We have
\begin{equation}\label{eq.ig-b-frobenius-lifts} \Ig_{b}=\varprojlim_n \Ig_b/p^n\mathbb{Z}_p(1) = \varprojlim (\Ig_\Mant \xleftarrow{F} \Ig_\Mant \xleftarrow{F} \ldots ). \end{equation}
We write also write $F$ for the associated endomorphism of $\mathbb{V}_{\Mant}$, so that we obtain a natural isomorphism 
\begin{equation}\label{eq.Vb-as-frobenius completion}
\mathbb{V}_b \cong \left(\varinjlim (\mathbb{V}_{\Mant} \xrightarrow{F}  \mathbb{V}_{\Mant}  \xrightarrow{F} \mathbb{V}_{\Mant} \xrightarrow{F} \ldots )\right)^\wedge
\end{equation}
The $k$th term in \cref{eq.Vb-as-frobenius completion} is identified with the $p^k\mathbb{Z}_p(1)$-invariants in $\mathbb{V}_\Mant$, and the density of the colimit is thus a special case of \cref{cor.density-zp1}. We note that \cref{eq.Vb-as-frobenius completion} is equivariant for the $T(\mathbb{Z}_p) \times \GL_2(\mathbb{A}_f^{(p)})$ actions on both sides. 

On each term in \cref{eq.Vb-as-frobenius completion}, one has a $C(\mathbb{Z}_p, \mathbb{Z}_p)$ action, and these are compatible along multiplication by $p$. In other words, if we view the $k$th copy of $\mathbb{V}_{\Mant}$ in \cref{eq.Vb-as-frobenius completion} as having an action of $C(1/p^k \mathbb{Z}_p, \mathbb{Z}_p)$, then these actions are compatible along restriction. We thus obtain an action of $C(\mathbb{Q}_p, \mathbb{Z}_p)=\varprojlim_k C(1/p^k \mathbb{Z}_p, \mathbb{Z}_p)$ on $\mathbb{V}_b$, which is precisely the one considered above in \cref{ss.action-of-cont-functions}.

\subsection{Reconstruction of $\mathbb{V}_\Mant$ from $\mathbb{V}_{\Katz'}$}\label{ss.vmant-from-vkatz}
In \cref{eq.Vb-as-frobenius completion} we recalled that $\mathbb{V}_b$ can be built up from $\mathbb{V}_{\Mant}$. We now recall that $\mathbb{V}_\Mant$ can be can be reconstructed from $\mathbb{V}_{\Katz'}$. Together, these relations will be used to deduce results for $\mathbb{V}_b$ from classical results on $\mathbb{V}_{\Katz'}$ in the theory of $p$-adic modular forms. 

There is a natural section $\Ig_{\Katz'} \rightarrow \Ig_{\Mant}$ defined as follows: choosing the elliptic curve $E$ in the prime-to-$p$ isogeny class with prime-to-$p$ Tate module corresponding to $(\hat{\mathbb{Z}}^{(p)})^2$ under $\varphi^{(p)}$, we can use the Weil pairing on $E[p^\infty]$ to obtain a trivialization $\varphi_p^\circ$ as the dual of $\varphi_p^\et$ (see \cite[4.4.2]{Howe.AUnipotentCircleActionOnpAdicModularForms} for a detailed description of the dual construction of a section from $\Ig_\Katz$). 
In particular, since $\Ig_{\Mant} \rightarrow \Ig_{\Katz'}$ is a $\mathbb{Z}_p^\times$-torsor, this section induces an isomorphism
\begin{equation}\label{eq.Vmant-as-functions} \mathbb{V}_{\Mant}=C(\mathbb{Z}_p^\times, \mathbb{V}_{\Katz'}).\end{equation} 
The section is equivariant for the action of $a \in \mathbb{Z}_p^\times$ on $\Ig_{\Katz'}$ when it is matched with the action of $(a^{-1}, a)$ on $\Ig_{\Mant}$, thus \cref{eq.Vmant-as-functions} is equivariant for the action of $T(\mathbb{Z}_p)$, where on the right 
\[ \left((a_1,a_2) \cdot f \right)(x)=\left( (a_1a_2, 1)(a_2^{-1},a_2) \cdot f\right)(x) = a_2 \cdot \left(f( (a_1a_2)^{-1} x)\right).\]
 It is furthermore equivariant for the action of $\Cont(\mathbb{Z}_p, \mathbb{Z}_p)$, when $\mathbb{V}_{\Katz'}$ is equipped with the action by viewing $\Ig_{\Katz'}$ as a $\mu_{p^\infty}$-invariant component of $\Ig_\Mant$ via the canonical section (this is essentially the construction of \cite{Howe.AUnipotentCircleActionOnpAdicModularForms}; in particular, as in \cite{Howe.AUnipotentCircleActionOnpAdicModularForms},  we emphasize that the quotient presentation of $\mathbb{V}_{\Katz'}$ in \cref{eq.quotient-presentation} does not itself give rise to an action of $\mu_{p^\infty}$ on $\Ig_{\Katz'}$!). \\

The identification \cref{eq.Vmant-as-functions} introduces a twist in the prime-to-$p$ action because the section $\Ig_{\Katz'} \rightarrow \Ig_{\Mant}$ 
depends on the choice of a fixed member of the prime-to-$p$ isogeny class: 
\begin{lemma}\label{lemma.vmant-as-functions-equivariance}
    \cref{eq.Vmant-as-functions} is $\GL_2(\mathbb{A}_f^{(p)})$-equivariant when $C(\mathbb{Z}_p^\times, \mathbb{V}_\Katz')$ is equipped with the action
    \[ \left(g \cdot f\right)(z)= (g \cdot f(z |\det(g)|)). \]
    where $|\det(g)|=\prod_{\ell \neq p} |\det(g_\ell)|_\ell$, viewed as an element of $\mathbb{Z}_p^\times$.   
\end{lemma}

\begin{remark}
We note that $|\det|$ in \cref{lemma.vmant-as-functions-equivariance} agrees with the restriction of $\tilde{\det}$ of \cref{sss.mf-twists} to $\GL_2(\mathbb{A}_f^{(p)})\leq \GL_2(\mathbb{A}_f)$ followed by projection $\hat{\mathbb{Z}}^\times \rightarrow \mathbb{Z}_p^\times$. 
Note that this twist also appears in the description of the group action in 
    \cite[Theorem A]{Howe.AUnipotentCircleActionOnpAdicModularForms}.
\end{remark}

\subsection{General coefficients}
For $R/\mathbb{Z}_p$ $p$-adically complete, we write \[ \mathbb{V}_{b,R} := \mathcal{O}(\Ig_b \times_{\Spf \mathbb{Z}_p} \Spf R)=\mathbb{V}_b \hat{\otimes}_{\mathbb{Z}_p} R = \lim_n \mathbb{V}_b \otimes_{\mathbb{Z}_p} R/p^n. \]
We use similar notation for $\mathbb{V}_\Mant, \mathbb{V}_\Katz,$ and $\mathbb{V}_{\Katz'}$. For $R=S[1/p]$ where $S/\mathbb{Z}_p$ is $p$-adically complete and $p$-torsion free (i.e. flat over $\mathbb{Z}_p$), we write $\mathbb{V}_{b,R}=\mathbb{V}_{b,S}[1/p]$. 

The constructions of the previous sections all apply working over any such $R$. In particular, on $\mathbb{V}_{b,R}$ we have an action of $C^\bdd(\mathbb{Q}_p, R)$, where $C^\bdd(\mathbb{Q}_p, R)$ denotes the continuous bounded functions to $R$ --- if $R$ is $p$-adically complete, then all continuous functions from $\mathbb{Q}_p$ to $R$ are bounded; if $R=S[1/p]$ for $S$ $p$-adically complete, then the bounded continuous functions can be written as $C(\mathbb{Q}_p, S)[1/p].$

\section{The completed Kirillov model}\label{s.completed-kirillov}
Let $\Qpab$ be the completion of the maximal abelian extension of $\mathbb{Q}_p$. In this section we explain how to interpret the cuspidal $p$-adic automorphic forms $\mathbb{V}_{b,\Qpab}^\cusp$ as a completion of the global Kirillov model of the space $S_k$ of cuspidal modular forms in any weight $k \geq 2$ (\cref{main.completion-kirillov-model}; see \cref{theorem.completion-kirillov-model} below). 

\subsection{The $q$-expansion and Kirillov maps}

\subsubsection{}
Recall the Tate curve $\Tate(q)$ from \cref{ss.tate-curve}. 
After base change to $\Spf \mathbb{Z}_p((q))^\wedge$ (where the completion is $p$-adic), we also have canonical isomorphisms
\[ \varphi^\circ_{p,\can}: \Tate(q)[p^\infty]^\circ \xrightarrow{\sim} \mu_{p^\infty} \textrm{ and } \varphi^\et_{p,\can}: \Tate(q)[p^\infty]^\et \xrightarrow{\sim} \mathbb{Q}_p/\mathbb{Z}_p. \]
After base change to $\Spf \mathbb{Z}_p((q^{\mathbb{Z}[1/p]_{>0}}))^\wedge$, these are induced by a canonical isomorphism
\[ \varphi_{p,\can}: \Tate(q)[p^\infty] \rightarrow \mu_{p^\infty} \times \bbQ_p/\bbZ_p. \]

\subsubsection{}

Let $\breve{\mathbb{Z}}_p=W(\overline{\mathbb{F}}_p)$. Fixing a trivialization of $\hat{\mathbb{Z}}^{(p)}(1)$ over $\Spf \breve{\mathbb{Z}}_p$ (equivalently, a map $\mathbb{Z}_{(p)}[ (\zeta_n)_{(n,p)=1}] \rightarrow \breve{\mathbb{Z}}_p$), we obtain a map 
\[ T(\mathbb{Q}_p) \times \GL_2(\mathbb{A}_f^{(p)}) \times \Spf \breve{\mathbb{Z}}_p((q^{\bbQ_{>0}}))^\wedge \rightarrow \Ig_b \textrm{ (for $\wedge$ the $p$-adic completion)} \]
given by the orbit for the group action on  
\[ (\Tate(q),  \varphi_{p,\can}, \varphi^{(p)}_\can). \]

Evaluating on global sections gives a total $q$-expansion map
\begin{equation} \label{eq.total-q-expansion} \mathbb{V}_b \rightarrow C(T(\mathbb{Q}_p) \times \GL_2(\mathbb{A}_f^{(p)}), \breve{\mathbb{Z}}_p((q^{\bbQ_{>0}}))^\wedge), \end{equation}
where the superscript $\wedge$ on the codomain denotes $p$-adic completion.

We write $\mathbb{V}_b^\mr{hol}$ for the preimage under \cref{eq.total-q-expansion} of functions valued in $\breve{\mathbb{Z}}_p[[q^{\bbQ_{>0}}]]^\wedge$, and $\mathbb{V}_b^{\mr{\cusp}} \subseteq \mathbb{V}^\mr{hol}_b$ for the submodule corresponding to series with constant term zero. 

We obtain a map $\Kir: \mathbb{V}^\cusp_b \rightarrow C(\mathbb{A}_f^\times, \Zpbreve)$ by applying the total $q$-expansion map (\ref{eq.total-q-expansion}) followed by restriction to 
\[ \begin{bmatrix} \mathbb{A}_f^\times & 0 \\ 0 & 1 \end{bmatrix} \leq T(\mathbb{Q}_p) \times \GL_2(\mathbb{A}_f^{(p)}) \]
 and then passing to the coefficient of $q$. 

Note that our choice of trivialization of $\hat{\mathbb{Z}}^{(p)}(1)$ can be viewed as a character $\tau^{(p)}: \mathbb{A}_f^{(p)} \rightarrow \Zpbreve^\times$ by the composition
\[ \mathbb{A}_f^{(p)} \rightarrow \mathbb{A}_f^{(p)}/\hat{\mathbb{Z}}^{(p)} =\mathbb{A}_f^{(p)}(1)/\hat{\mathbb{Z}}^{(p)}(1) = \mu(\Zpbreve). \]
Using this character, we equip $C(\mathbb{A}_f^\times, \Zpbreve)$
with an action of
\[ \mathbb{A}_f^\times \ltimes \left( C(\mathbb{Q}_p, \Zpbreve) \times (\mathbb{A}_f^{(p)})^\times \right)\]
by 
\[ ((a, f, u) \cdot h) (x) =  h(ax)f(x_p) \tau^{(p)}(u x^{(p)}). \]

\begin{lemma}\label{lemma.kir-equivariant}The action of $\left(T(\mathbb{Q}_p) \ltimes C(\mathbb{Q}_p, \Zpbreve)\right)\times \GL_2(\mathbb{A}_f^{(p)})$ on $\mathbb{V}_b$ preserves $\mathbb{V}^\cusp_b$, and $\Kir$ is equivariant for the actions of $\mathbb{A}_f^\times \ltimes \left( C(\mathbb{Q}_p, \Zpbreve) \times (\mathbb{A}_f^{(p)})^\times \right)$.
\end{lemma}
\begin{proof}
That the cuspidal part is preserved by the action of $T(\mathbb{Q}_p) \times \GL_2(\mathbb{A}_f^{(p)})$ is automatic from the definition of the total $q$-expansion map \cref{eq.total-q-expansion}, and the $\mathbb{A}_f^\times$-equivariance of $\Kir$ is built into the definition. Finally, the preservation of the cuspidal part and equivariance for the action of $C(\mathbb{Q}_p, \Zpbreve) \times (\mathbb{A}_f^{(p)})^\times $ can be established by an explicit computation on $q$-expansions as in \cite[\S6.5]{Howe.AUnipotentCircleActionOnpAdicModularForms} (cf. \cref{remark.classical-q-expansions}). 
\end{proof}

\begin{remark}\label{remark.classical-q-expansions}
Let $N$ be coprime to $p$. For $K^{(p)}_1(N)$ as in \cref{sss.congruence-subgroups}, let \[ K_{b}^1(N) := \begin{bmatrix} \mathbb{Z}_p^\times & \mathbb{Z}_p(1) \\ 0 & 1 \end{bmatrix}\times K^{(p)}_1(N). \]
We note that $\mathbb{A}_f^\times = \mathbb{Q}_{>0} \times \hat{\mathbb{Z}}^\times.$ 
From the equivariance it is immediate that the restriction of $\Kir$ to $(\mathbb{V}^{\cusp}_b)^{K_b^1(N)} \subseteq \mathbb{V}^\cusp_{\Katz'}$ factors through $C(\mathbb{Z}_{>0}, \breve{\mbb{Z}}_p)$ (in fact, even $C(\mathbb{Z}_{>0}, \mathbb{Z}_p)$), and the value at $n \in \mathbb{Z}_{>0}$ is simply the coefficient of $q^n$ in the usual $q$-expansion of a cuspidal $p$-adic modular form of prime-to-$p$ level $\Gamma_1(N)$.  
\end{remark}

\begin{proposition}\label{prop.Kir-q-exp-principle}
The map $\Kir: \mathbb{V}^\cusp_{b,\Zpbreve} \rightarrow C(\mathbb{A}_f^\times, \Zpbreve)$ is an injection with $\Zpbreve$-flat cokernel. 
\end{proposition}
\begin{proof}
We argue the injectivity with flat cokernel in essentialy the same way as the $q$-expansion principle of Katz \cite[Lemma on bottom of p.499]{Katz.p-adic-L-functions-via-moduli-of-elliptic-curves}. We first observe that it will suffice to establish injectivity of $\Kir \otimes \overline{\mathbb{F}}_p$: let us assume this injectivity and explain how to deduce the other statements. Note that being $\Zpbreve$-flat is equivalent to being $p$-torsion free so, since $C(\mathbb{A}_f^\times, \Zpbreve)$ is $p$-torsion free, the image of $\Kir$ is $p$-torsion free. Thus applying the $\mathrm{Tor}_{\Zpbreve}^\bullet(-, \overline{\mathbb{F}}_p)$-sequence to 
\[ 0 \rightarrow \ker\, \Kir \rightarrow \mathbb{V}^\cusp_{b,\Zpbreve} \rightarrow \Im \Kir \rightarrow 0, \]
we find that 
\[ (\ker \Kir) \otimes \overline{\mathbb{F}}_p = \ker (\Kir \otimes \overline{\mathbb{F}}_p)\]
and thus by assumption is zero. But $\ker \Kir$ is $p$-adically complete as the kernel of a map of $p$-adically complete modules, so this implies also that $\ker \Kir=0$.  Now applying the 
$\mathrm{Tor}_{\Zpbreve}^\bullet(-, \overline{\mathbb{F}}_p)$-sequence to the exact sequence
\[ 0 \rightarrow \mathbb{V}^\cusp_{b,\Zpbreve} \xrightarrow{\Kir} C(\mathbb{A}_f^\times, \Zpbreve) \rightarrow \mathrm{coker}(\Kir) \rightarrow 0 \]
we find $\mathrm{Tor}^1(\mathrm{coker}(\Kir), \overline{\mathbb{F}}_p)=\ker (\Kir \otimes \overline{\mathbb{F}_p})$ and thus is zero, so that $\mathrm{coker}(\Kir)$ is $p$-torsion free.  

It remains to establish the injectivity modulo $p$. This follows by an extension of the argument in \cite[Start of section XI]{Katz.p-adic-L-functions-via-moduli-of-elliptic-curves}: the action of $\mathbb{Q}_{>0}^\times \leq \mathbb{A}_f^\times = \mathbb{Q}_{>0}^\times \times \hat{\mathbb{Z}}^\times $ can be used to translate any $q$-expansion coefficient to the first one as encoded by $\Kir$. We will argue below that $\Kir$ encodes the $q$-expansion at a cusp in each connected component of $\Ig^b_{\overline{\mathbb{F}}_p}$, giving the injectivity. 

To see that each connected component is hit, we show the action of $\hat{\mathbb{Z}}^\times$ is transitive on the connected components of $\Ig^b_{\overline{\mathbb{F}}_p}$: This transitivity holds on $\Ig_{\Mant,{\overline{\mathbb{F}}_p}}$ by combining \cref{eq.Vmant-as-functions} with the fact that, for each $N$ coprime to $p$,
$\Ig_{\Katz',\overline{\mathbb{F}}_p}/{K^{(p)}(N)}$ has the same connected components as the full mod $p$ modular curve of level $K^{(p)}(N)$ over $\overline{\mathbb{F}}_p$ (as in  \cite[Remark on p.499]{Katz.p-adic-L-functions-via-moduli-of-elliptic-curves}). Indeed, these latter are identifed with $\mu_N(\overline{\mathbb{F}}_p)^\times$ by the Weil pairing and thus they transform by the natural transitive action of $\hat{\mathbb{Z}}^\times$ through $(\mathbb{Z}/N\mathbb{Z})^\times$ (cf. the computation over $\mathbb{Q}$ in \cref{sss.mf-twists}). This description of the connected components  persists to $\Ig_{b,\overline{\mbb{F}_p}}$ since the inseparable maps in \cref{eq.ig-b-frobenius-lifts} do not change the connected components. 
\end{proof}

\subsection{Evaluation of modular forms}\label{ss.evaluation}

\subsubsection{}We first recall the usual formalism for evaluation of modular forms to $p$-adic modular forms but with notation adapted to our setting. The universal elliptic curve over $\Ig_{\Katz'}$ is equipped with a canonical basis $\omega_\can$ for $\omega_{E^\vee}$: the isomorphism $\varphi_p^\et: T_p E[p^\infty]^\et \cong \mathbb{Z}_p$ is equivalent to an isomorphism $E^\vee[p^\infty]^\circ \xrightarrow{\sim} \mu_{p^\infty}$, and we pullback the invariant differential $\frac{dt}{t}$ on $\mu_{p^\infty}$ along this map to obtain $\omega_\can$. If we algebraize to an elliptic curve over $\Spec \mathbb{V}_{\Katz}'$ and invert $p$, then, for 
\[  K_1(p^\infty)=\begin{bmatrix} \mathbb{Z}_p^\times & \mathbb{Z}_p \\ 0 & 1\end{bmatrix} \leq \GL_2(\mathbb{Z}_p)\]
we also have level $K_1(p^\infty)$ structure coming from $\varphi_p^\et$, and thus we obtain a classifying map to the associated infinite level scheme-theoretic modular curve over $\mathbb{Q}_p$ (see \cref{sss.modular-curves}),
\begin{equation}
\label{eq.classifying-map-classical}	
 \Spec \mathbb{V}_{\Katz',\mathbb{Q}_p} \rightarrow Y_{K_1(p^\infty) \times \{1\}, \mathbb{Q}_p} 
\end{equation}
where $\{1\}$ is the trivial subgroup of $\GL_2(\mathbb{A}_f^{(p)})$ . 

Pulling back sections along \cref{eq.classifying-map-classical} and evaluating on $\omega_\can$, we obtain, for each $k \in \mathbb{Z}$, a map 
\begin{equation}
\label{eq.evaluation-classical-all-sections}	
H^0(Y_{K_1(p^\infty)\times \{1\}, \mathbb{Q}_p}, \omega_{E^\vee}^k) \rightarrow  \mathbb{V}_{\Katz', \mathbb{Q}_p}. 
\end{equation}
It is $\GL_2(\mathbb{A}_f^{(p)})$-equivariant by construction; it is also equivariant for the induced actions of $1 \times \mathbb{Z}_p^\times$, so long as on the left we twist the natural smooth action by $(1, z) \mapsto z^{-k}$ (coming from the non-trivial action on $\omega_\can$). Moreover, the map is compatible with $q$-expansions, so it is an injection and induces 
\begin{equation}\label{eq.classical-eval-map} S_{k,\mathbb{Q}_p}^{K_1(p^\infty)} \hookrightarrow \mathbb{V}_{\Katz', \mathbb{Q}_p}^{\cusp}.\end{equation} 

\newcommand{\Zpab}{\mathbb{Z}_p^\ab}
\newcommand{\Zpcyc}{\mathbb{Z}_p^\cyc}
\newcommand{\Qpcyc}{\mathbb{Q}_p^\cyc}

\subsubsection{}
If we base change to $\mathbb{Z}_p^\cyc$, the completion of the ring of integers in $\mathbb{Q}_p(\mu_{p^\infty}(\mathbb{C}_p))$, and fix also a trivialization of $\mathbb{Z}_p(1)=T_p(\mu_{p^\infty})$ over $\mathbb{Q}_p^\cyc$ (i.e. a compatible system of $p$-power roots of unity in $\mathbb{C}_p$), then we obtain a classifying map 
\begin{equation}\label{eq.classifying-map-full-level}\Spec \mathbb{V}_{b, \mathbb{Q}_p^\cyc} \rightarrow Y_{\{1\}, \mathbb{Q}_p} \end{equation}
where $Y_{\{1\}}$ is the scheme-theoretic modular curve at full infinite level (i.e. $\{1\}$ denotes the trivial subgroup of $\GL_2(\mathbb{A}_f)$): for the level structure at $p$, we compose the map induced by $\varphi_p$, $V_p  E \xrightarrow{\sim} V_p(\mu_{p^\infty} \times \mathbb{Q}_p/\mathbb{Z}_p) $ with the trivialization  
\[ V_p (\mu_{p^\infty} \times \mathbb{Q}_p/\mathbb{Z}_p)= \mathbb{Q}_p(1) \times \mathbb{Q}_p = \mathbb{Q}_p^2\]
 over $\Spec \mathbb{Q}_p^\cyc$ obtained from our fixed trivialization of $\mathbb{Z}_p(1)$ over $\mathbb{Q}_p^\cyc$. Pulling back along \cref{eq.classifying-map-full-level} and evaluating on $\omega_\can$, we obtain an extension of \cref{eq.classical-eval-map}, 
\begin{equation}\label{eq.full-eval-map-all-sections}  H^0(Y_{\{1\}, \mathbb{Q}_p^\cyc}, \omega_{E^\vee}^k) \hookrightarrow  \mathbb{V}_{\Katz', \Qpcyc}. \end{equation}

\subsubsection{} Again \cref{eq.full-eval-map-all-sections} is $\GL_2(\mathbb{A}_f^{(p)})$-equivariant by construction. It is also $B(\mathbb{Q}_p)$-equivariant when we map this group into $\tilde{G}_b=\begin{bmatrix} \mathbb{Q}_p^\times & \tilde{\mu}_{p^\infty} \\ 0 & \mathbb{Q}_p^\times \end{bmatrix}$ by the identity on the diagonal and, in the upper right corner, by using the map  induced by the trivialization of $\mathbb{Z}_p(1)$,
\[ \mathbb{Q}_p=\mathbb{Z}_p[1/p]=\mathbb{Z}_p(1)[1/p]\hookrightarrow \tilde{\mu}_{p^\infty}(\mathbb{Z}_p^\cyc), \]
 so long as we twist the domain of \cref{eq.full-eval-map-all-sections} by the character $\begin{bmatrix}z_1 & u \\ 0 & z_2 \end{bmatrix} \mapsto z_2^{-k}$, again because of the non-trivial action on $\omega_\can$. 
 
 The map \cref{eq.full-eval-map-all-sections} is compatible with $q$-expansions, so it is an injection and induces an equivariant map
\begin{equation}\label{eq.eval-map-def} \eval_k: S_{k, \Qpcyc} \hookrightarrow \mathbb{V}_{b,\Qpcyc}^{\cusp}. \end{equation}
The map of \cref{eq.classical-eval-map} is simply the restriction of $\eval_k$ to the $K_1(p^\infty)$-invariants (and to $S_{k,\mbb{Q}_p} \subseteq S_{k,\Qpcyc}$).

\begin{lemma}\label{lemma.eval-dense} For any $k \geq 2$, the image of $S_{k,\Qpcyc}$ under $\eval_k$ is dense in $\mathbb{V}_{b, \Qpcyc}^\cusp$.  	
\end{lemma}
\begin{proof}
	For $N$ coprime to $p$,  $S_{k,\Qpcyc}^{B_1(\mathbb{Z}_p)K_1^{(p)}(N)}$ is the union over $r \geq 0$ of the spaces of classical cusp forms of level $\Gamma_1(p^r N)$.  The density of 
\[ S_k^{B_1(\mathbb{Z}_p)K_1^{(p)}(N)} \hookrightarrow \mathbb{V}_{\Katz', \Qpcyc}^{\cusp,K_1^{(p)}(N)} = \mathbb{V}_{\Katz', \Zpcyc}^{\cusp,K_1^{(p)}(N)}[1/p]\]
is then a result of Hida using the duality between $p$-adic modular forms and their completed Hecke algebras and a comparison of the completed Hecke algebras through Eichler-Shimura obtained by an argument with singular homology (the specific density result is claimed, e.g., near the top of \cite[p. 551]{Hida.GaloisRepresentationsIntoGL2ZPX} and attributed to Shimura in the Hecke algebra formulation in \cite[Immediately above Theorem 2.1]{Hida.OnPadicHeckeAlgebrasForGL2}; see also \cite[Theorem III.3.2]{Gouvea.ArithmeticOfpAdicModularForms}). 

We obtain identifications from \cref{eq.Vmant-as-functions} and \cref{lemma.vmant-as-functions-equivariance}, 
\[ \left(\mathbb{V}_{\Mant, \Qpcyc}^\cusp \right)^{K_1^{(p)}(N)} = C\left(\mathbb{Z}_p^\times, \left(\mathbb{V}_{\Katz', \Qpcyc}^\cusp \right)^{K_1^{(p)}(N)}\right). \]
Thus from the density established above we deduce also the density of the image of 
\[  S_{k, \Qpcyc}^{U(\mathbb{Z}_p)K_1^{(p)}(N)} \rightarrow \left(\mathbb{V}_{\Mant, \Qpcyc}^\cusp \right)^{K_1^{(p)}(N)}. \]
where $U(\mathbb{Z}_p)$ is the upper triangular unipotent subgroup. Indeed, the closure of the image contains any locally constant function (because we are also introducing the same $\mathbb{Z}_p^\times$-worth of connected components when we pass to the modular curve $Y_{U(\mathbb{Z}_p)K_1^{(p)}(N), \Qpcyc}$; cf. the proof of \cref{prop.Kir-q-exp-principle}), and these are dense. 

Writing 
\[ U_n:=\begin{bmatrix}1 & p^n \mathbb{Z}_p \\ 0& 1\end{bmatrix}, \] and applying the equivariance of the evaluation map to the action of $\begin{bmatrix}p & 0\\ 0& 1\end{bmatrix}$, we obtain the density of $S_{k, \Qpcyc}^{U_nK_1^{(p)}(N)}$ in $\left(\mathbb{V}_{b,\Qpcyc}^{\cusp}\right)^{p^n\mathbb{Z}_p(1) \times K_1^{(p)}(N)}.$  The density of the image of $S_{k,\Qpcyc}^{K_1^{(p)}(N)}$ in $(\mathbb{V}_b^\cusp)^{K_1^{(p)}(N)}$ then follows from \cref{cor.density-zp1} (or from the more explicit form in \cref{eq.Vb-as-frobenius completion} and surrounding discussion). 

The argument to pass from density at level $K_1^{(p)}(N)$ to density at level $K^{(p)}(N)$ (matrices congruent to the identity mod $N$) is similar but more elementary: for any $N$ prime-to-$p$, one passes from level $K_1^{(p)}(N)$ to the matrices $K_2^{(p)}(N)$ that are upper triangular unipotent mod $N$ by considering the connected components then uses the action of $\begin{bmatrix} N & 0 \\  0 & 1 \end{bmatrix}$,
to conjugate $K_2^{(p)}(N^2)$ into a subgroup of $K^{(p)}(N)$.

Finally, since the vectors preserved by $K^{(p)}(N)$ for some $N$ (i.e. the smooth vectors for the $\GL_2(\mathbb{A}_f^{(p)})$-action)  are dense in $\mathbb{V}_{b,\Qpcyc}^{\cusp}$ (because $\Ig^b=\varprojlim_N \Ig^b/K^{(p)}(N)$ and each term is an affine $p$-adic formal scheme), we obtain the result. 
\end{proof}

\subsection{The completed Kirillov model}

Note that $\Qpab$, the completion of the maximal abelian extension of $\mathbb{Q}_p$ in $\mathbb{C}_p$, contains $\Qpcyc$ and $\Qpbreve$. Thus, after base change to $\Qpab$ and making a choice of trivialization of $\hat{\mathbb{Z}}(1)=\mathbb{Z}_p(1) \times \hat{\mathbb{Z}}^{(p)}(1)$, the maps $\eval_k$ and the map $\Kir$ are all defined. 

We may view our choice of a trivialization of $\hat{\mathbb{Z}}(1)$ as the choice of a non-degenerate locally constant $\Qpab$-valued character $\tau$ of $\mathbb{A}_f^\times$, 
\[ \tau: z \mapsto \overline{z} \in \mathbb{A}_f^\times/\hat{\mathbb{Z}}^\times = \mathbb{A}_f^\times(1)/\hat{\mathbb{Z}}^\times(1).\]
We obtain an associated Kirillov action of $\begin{bmatrix}\mathbb{A}_f^\times & \mathbb{A}_f \\ 0 & 1\end{bmatrix}$ on $C^\bdd(\mathbb{A}_f^\times, \Qpcyc)$ by 
\[ \left(\begin{bmatrix} a & u \\ 0 & 1  \end{bmatrix} \cdot f\right) (z)=f(az)\tau(uz)\]

\begin{theorem}\label{theorem.completion-kirillov-model}
The composition $\Kir \circ \eval_k$ is an $\begin{bmatrix}\mathbb{A}_f^\times & \mathbb{A}_f \\ 0 & 1\end{bmatrix}$-equivariant injection for the Kirillov action on $C^\bdd(\mathbb{A}_f^\times, \Qpcyc)$ associated to $\tau$. For any $k \geq 2$, $\Kir \circ \eval_k$ identifies $\mathbb{V}_{b, \Qpab}^\cusp$ with the completion of this Kirillov model inside $C^\bdd(\mathbb{A}_f^\times, \mathbb{Q}_p^\ab)$ (which is equipped with its Banach topology via the sup norm). 
\end{theorem}
\begin{proof}
The equivariance follows from \cref{lemma.kir-equivariant} and the description of the equivariance at $p$ for $\eval_k$ given in \cref{ss.evaluation} 
(note, in particular, that the twist at $p$ uses the bottom right diagonal entry, so is invisible to the group action we are considering here). Since $\Kir$ is an isometry onto its closed image (by the injectivity and flatness of the cokernel in \cref{prop.Kir-q-exp-principle}), the rest follows from \cref{lemma.eval-dense}. 
\end{proof}

\begin{remark}  
If instead of fixing a single $k$ we consider $\bigoplus \eval_k$, then one can obtain the density of the image by imitiating the argument in \cref{lemma.eval-dense} starting from Katz's density result for modular forms of level $1$ at $p$ but arbitrary weight \cite{Katz.HigherCongruencesBetweenModularForms}. The method of \cite{Howe.ThepAdicJacquetLanglandsCorrespondenceAndAQuestionOfSerre} provides a different argument to show the image of $\bigoplus \eval_k$ is dense that is more technical but very soft in the sense that it can be applied to establish analogous statements for many more Igusa varieties. 
\end{remark}

\section{Coinvariants and ordinary forms}\label{s.coinvariants-ordinary-lg}

In this section we prove \cref{main.hida-coinvariants} and \cref{maincor.admissibility}.

\subsection{The $U_p$-operator}\label{ss.Up}

We consider the $U_p$-operator on $\mathbb{V}^\cusp_{\Mant, \mathbb{Q}_p}$ obtained as the composition of 
\[ \mathbb{V}^\cusp_{\Mant}= \mathbb{V}^{\cusp, \mathbb{Z}_p(1)}_{b}  \xrightarrow{ \begin{bmatrix} p & 0 \\ 0 & 1 \end{bmatrix} \cdot } \mathbb{V}^{\cusp, p\mathbb{Z}_p(1)}_{b} \xrightarrow{1_{\mathbb{Z}_p} \cdot } \mathbb{V}^{\cusp}_{\Mant} \]
where $1_{\mathbb{Z}_p}$ denotes the indicator function of $\mathbb{Z}_p$ in $C(\mathbb{Q}_p, \mathbb{Z}_p)$. On $q$-expansions it acts as the usual $U_p$ operator in the theory of $p$-adic modular forms: noting $\mathbb{V}^\cusp_\Mant=(\mathbb{V}_b^\cusp)^{\mathbb{Z}_p(1)}$, the equivariance of $\Kir$ implies it maps $\mathbb{V}^\cusp_\Mant$ into
\[ C^\bdd(\mathbb{A}_f^\times, \Qpbreve)^{\mathbb{Z}_p(1)}=C^\bdd(\mathbb{Z}_p^\times \times p^{\mathbb{Z}_{\geq 0}} \times (\mathbb{A}_f)^\times, \Qpbreve) = C^\bdd(\mathbb{Z}_p\backslash 0 \times (\mathbb{A}_f)^\times, \Qpbreve), \]
and, by an immediate computation, $U_p$ acts in this model by \begin{equation}\label{eq.Up-pullback-by-p}(U_p \cdot f)(x)=f(px).\end{equation} 

\begin{remark}\label{remark.classical-up}
The map $\eval_k$ identifies the double coset operator
\[ U_p=\frac{1}{p} \left( U(\mathbb{Z}_p)\begin{bmatrix} p & 0 \\ 0 & 1 \end{bmatrix}U(\mathbb{Z}_p) \right) = \frac{1}{p}\sum_{i=0}^{p-i} \begin{bmatrix}p & i \\ 0 & 1\end{bmatrix}U(\mathbb{Z}_p) \]
on $S_k^{U(\mbb{Z}_p)}$ with the operator $U_p$ defined above (cf. also \cref{eq.Uell-double-coset}). Indeed, this follows from the equivariance properties of $\eval_k$ by expressing the indicator function $1_{\mathbb{Z}_p}$ as $\frac{1}{p}$ times the sum over all characters of $\frac{1}{p}\mbb{Z}_p/\mbb{Z}_p$.
\end{remark}

\subsection{The ordinary projector}

We write $e:=\lim_n U_p^{n!}$ for Hida's ordinary projector on $\mathbb{V}^\cusp_\Mant$, which exists by an analog of the usual argument, recalled now. 

\begin{lemma}
	$\lim_n U_p^{n!}$ exists as an idempotent operator on $\mathbb{V}^\cusp_\Mant$. 
\end{lemma}
\begin{proof}
By the proof of \cref{lemma.eval-dense}, the finite dimensional spaces of classical cuspidal modular forms $S_{k,\mathbb{C}_p}^{K_pK^p}$ are dense in $\mathbb{V}^\cusp_{\Mant, \mathbb{C}_p}$ as we vary over all $K^p \leq \GL_2(\mathbb{A}_f^{(p)})$ compact open and, for all $r \geq 0$, $K_p\leq \GL_2(\mathbb{Z}_p)$ consisting of those matrices congruent modulo $p^r$ to an element in $U(\mathbb{Z}_p)$. By \cref{remark.classical-up}, the $U_p$ operator preserves each of these spaces. It thus suffices to show the limit exists on each such space. 

On $S_{k,\mathbb{C}_p}^{K_pK^p}$, we claim $e$ acts as projection onto the sum of generalized $U_p$-eigenspaces corresponding to eigenvalues of absolute value $1$. Indeed, since $U_p$ preserves the $\mathcal{O}_{\mathbb{C}_p}$-lattice $S_{k,\mathbb{C}_p}^{K_pK^p} \cap \mathbb{V}^\cusp_{\Mant,\mathcal{O}_{\mathbb{C}_p}}$, its eigenvalues all have absolute value less than or equal to $1$. The limit is thus zero on each non-unit generalized eigenspace. On each unit generalized eigenspace, the limit is the identity because the eigenvalues $\lambda$ of $U_p$ are algebraic over $\mathbb{Q}_p$ (since $U_p$ is defined already on $S_k$ without base change to $\mathbb{C}_p$), so that $\lim_{n\rightarrow \infty} \lambda^{n!}=1$ is verified modulo $p^k$ by working in the finite abelian group $\left(\mathcal{O}_{\mathbb{Q}_p(\lambda)}/p^k\mathcal{O}_{\mathbb{Q}_p(\lambda)}\right)^\times$.
	
\end{proof}

\subsection{Topological coinvariants}
Recall from \cref{def.top-coinv-mupinfty} and \cref{def.top-coinv-tildemupinfty} that we define the topopological co-invariants
\[ (\mathbb{V}_b^{\cusp})_{\tilde{\mu}_{p^\infty}}:=\mathbb{V}_b^\cusp \hat{\otimes}_{C(\mathbb{Q}_p, \mathbb{Z}_p)} \mathbb{Z}_p \textrm{ and } (\mathbb{V}_\Mant^{\cusp})_{\mu_{p^\infty}}:=\mathbb{V}_\Mant^\cusp \hat{\otimes}_{C(\mathbb{Z}_p, \mathbb{Z}_p)} \mathbb{Z}_p \]
where $\mathbb{Z}_p$ is the fiber at $0$ as a $C(\mathbb{Q}_p,\mathbb{Z}_p)$-module, i.e. $f \cdot z = f(0)z$, and the completion is $p$-adic. Equivalently, we have
 \begin{align*} (\mathbb{V}_b^{\cusp})_{\tilde{\mu}_{p^\infty}}&= \lim_{k} \mathbb{V}_{b,\mathbb{Z}/p^k\mathbb{Z}}^\cusp \otimes_{C(\mathbb{Q}_p, \mathbb{Z}/p^k\mathbb{Z})} \mathbb{Z}/p^k\mathbb{Z} \textrm{ and } \\
 (\mathbb{V}_{\Mant}^{\cusp})_{\mu_{p^\infty}}& = \lim_{k} \mathbb{V}_{\Mant,\mathbb{Z}/p^k\mathbb{Z}}^\cusp \otimes_{C(\mathbb{Z}_p, \mathbb{Z}/p^k\mathbb{Z})} \mathbb{Z}/p^k\mathbb{Z}.
\end{align*}

The following result implies \cref{main.hida-coinvariants} of the introduction. 
\begin{theorem}\label{theorem.ordinary-iso-coinvariants}
The following natural maps are isomorphisms: 
\[ e \mathbb{V}_\Mant^{\cusp} \rightarrow  (\mathbb{V}_\Mant^{\cusp})_{\mu_{p^\infty}} \rightarrow (\mathbb{V}_b^\cusp)_{\tilde{\mu}_{p^\infty}} \] 
\end{theorem}
\begin{proof}
The second map is an isomorphism by \cref{lemma.topological-coinvariants-agree}. Thus it remains to treat the map $e \mathbb{V}_\Mant^{\cusp} \rightarrow  (\mathbb{V}_\Mant^{\cusp})_{\mu_{p^\infty}}$. Working mod $p^k$, and using the description of $U_p$ given in \cref{eq.Up-pullback-by-p}, it is clear that $v \in \mathbb{V}_{\Mant,\mathbb{Z}_p/p^k\mathbb{Z}_p}^\cusp$ is in the kernel of $e$ if and only if $\Kir(v)$, viewed as an element of 
\[ C(\mathbb{Z}_p\backslash\{0\} \times \mathbb{A}_f^\times, \mathbb{Z}/p^k\mathbb{Z}), \]
 vanishes on $p^n\mathbb{Z}_p\backslash\{0\} \times \mathbb{A}_f^\times$ for $n \gg 0$. This is equivalent to $v$ being in the
  the image of multiplication by the indicator function $1_{\mathbb{Z}_p \backslash p^n \mathbb{Z}_p}$ for $n \gg 0$. Since any element in the kernel of evaluation at $0$ in $C(\mathbb{Z}_p, \mathbb{Z}_p/p^k\mathbb{Z}_p)$ is a multiple of such an indicator function, we conclude the kernel of $e$ is equal to the kernel of passing to the fiber at $0$ as in the definition of the coinvariants. 
 \end{proof}

\subsection{Admissibility}

Combining the following result, deduced from Hida's finiteness theorem for the ordinary Hecke algebra, with \cref{theorem.ordinary-iso-coinvariants} gives \cref{maincor.admissibility}. 
\begin{proposition}\label{prop.admisibillity}
    For $K^p \leq \GL_2(\mathbb{A}_f^{(p)})$ compact open, the space $e\mathbb{V}_{\Mant,\mathbb{Q}_p}^{\cusp, K^p}$ is an admissible Banach representation of $T(\mathbb{Z}_p)$.
\end{proposition}
\begin{proof}
We first assume $K^p=K_1^{(p)}(N)$ for $N$ coprime to $p$.  We note that, in the definition of the ordinary part, we can replace $U_p$ with the operator $U_p'$ obtained as the composition of 
\[ \mathbb{V}^\cusp_{\Mant}= \mathbb{V}^{\cusp, \mathbb{Z}_p(1)}_{b}  \xrightarrow{ \begin{bmatrix} p & 0 \\ 0 & p^{-1} \end{bmatrix} \cdot } \mathbb{V}^{\cusp, p^2\mathbb{Z}_p(1)}_{b} \xrightarrow{1_{\mathbb{Z}_p} \cdot } \mathbb{V}^{\cusp}_{\Mant}. \]
The operator $U_p'$ is induced by an operator on $\mathbb{V}_{\Katz'}^\cusp$ under the identification of \cref{eq.Vmant-as-functions}. Using this construction of the operator $e$ and applying also \cref{lemma.vmant-as-functions-equivariance} we find that, in \cref{eq.Vmant-as-functions}, we have an identification 
\[ e\mathbb{V}^{\cusp,K^p}_{\Mant}=C(\mathbb{Z}_p^\times, e \mathbb{V}_{\Katz'}^{\cusp,K^p}) \]
of $T(\mathbb{Z}_p)$-representations, i.e. $e \mathbb{V}_\Mant^\cusp$ is induced from the $\mathbb{Z}_p^\times$ representation $e \mathbb{V}_{\Katz'}$, when $\mathbb{Z}_p^\times$ is included into $T(\mathbb{Z}_p)$ by $a \mapsto (a,a^{-1})$. It thus suffices to show $e\mathbb{V}_{\Katz'}^\cusp$ is an admissible representation of $\mathbb{Z}_p^\times$. By the duality between ordinary cuspidal $p$-adic modular forms and ordinary Hecke algebras (\cite[Equation (2.3) on p.437] {Hida.OnPadicHeckeAlgebrasForGL2}; see also \cite[\S3.3]{Hida.pAdicAutomorphicFormsOnShimuraVarieties}), it suffices to show the classical ordinary Hecke algebra with $p$ inverted is a finite rank $\mathbb{Z}_p[[\mathbb{Z}_p^\times]][1/p]$-module. This holds already without inverting $p$  (\cite[Theorem 2.2]{Hida.OnPadicHeckeAlgebrasForGL2}; see also \cite[\S3.3]{Hida.pAdicAutomorphicFormsOnShimuraVarieties}). 

To pass the statement from level $K_1^{(p)}(N)$ to level $K^{(p)}(N)$, we argue as in the end of the proof of \cref{lemma.eval-dense}: passing to level $K_2^{(p)}(N)$ gives, at the level of $T(\mathbb{Z}_p)$-representations, just a direct sum of $\phi(N)$ copies corresponding to the connected components. We then conjugate $K_2^{(p)}(N^2)$ into $K_1(N)$ for $N\gg0$ such that $K_1^{(p)}(N) \leq K^p$ to conclude (since the $K^p$ invariants are then a closed subrepresentation of the $K_1^{(p)}(N)$-invariants, which we have just argued are admissible). 
\end{proof}

\begin{remark}
  The incarnation of Hida-ordinary $p$-adic modular forms as $\tilde{\mu}_{p^\infty}$-coinvariants is closely related, via \cref{main.completion-kirillov-model}, to the theory of the canonical lift for the Jacquet module of $S_k$, which is smooth admissible because $S_k$ is smooth admissible. Because, for $k \geq 2$, \cref{main.completion-kirillov-model} implies that the induced map $J(S_{k, \mathbb{C}_p}) \rightarrow (\mathbb{V}_{b,\mathbb{C}_p}^\cusp)_{\tilde{\mu}_{p^\infty}}$ has dense image, it seems plausible that one may be able to deduce, for $K^p \leq \GL_2(\mathbb{A}_f^{(p)})$ compact open,  the admissibility of $(\mathbb{V}_{b,\mathbb{C}_p}^\cusp)^{K^p}_{\tilde{\mu}_{p^\infty}}$ as a Banach $T(\mathbb{Q}_p)$-representation directly from the admissibility of $J(S_{k,\mathbb{C}_p}^{K^p})$ as a smooth $T(\mathbb{Q}_p)$-representation. This would give a very direct deduction of Hida's finiteness theorem for the ordinary Hecke algebra from admissibility of the Jacquet module of a smooth admissible representation.  It would also be interesting to know if there is an  argument for Hida's classicality theorem from this perspective. 
\end{remark}

\section{Local-global compatibility for the $\GL_2$ ordinary Igusa variety}\label{s.local-global-proof}

In this section we prove \cref{main.lg}, describing the local representation in $\mathbb{V}_{b,\mathbb{C}_p}^\cusp$ associated to a global modular representation $\pi$. In \cref{ss.Galois-reps-ordinary-forms} we first recall some results on ordinary $p$-adic modular forms and their associated Galois representations. In \cref{ss.embedding-and-first-lower-bound} we construct a Kirillov embedding of this local representation and establish a first lower bound on its image using the completion of the smooth Kirillov model. In \cref{ss.upper-bound}, we establish the upper bound on the size of the representation using the admissibility of \cref{maincor.admissibility}. In \cref{ss.lower-bound-companion} we complete the lower bound, with the key input being the the theory of overconvergent companions in \cite{BreuilEmerton.RepresentationPAdiquesOrdinarsDeGl2Qp}, which we use to find the vector that is missing from the completion of the smooth Kirillov model when $\rho_p$ is reducible and split.  

\subsection{Galois representations attached to ordinary $p$-adic modular forms}\label{ss.Galois-reps-ordinary-forms}
We recall that for any prime-to-$Np$ Hecke eigenform in $f \in \mathbb{V}_{\Katz',\mathbb{C}_p}^{\cusp, K_1^{(p)}(N)}$, there is an associated semisimple Galois representation $\rho: \Gal(\overline{\mathbb{Q}}/\mathbb{Q}) \rightarrow \GL_2(\mathbb{C}_p)$, interpolating the construction for classical forms recalled in \cref{ss.galois-classical-modular-forms}. The following result of Kisin generalizes \cref{lemma.classical-ordinary}:

\begin{proposition}\label{prop.ordinary-p-adicmf-unramified-sub}
    Suppose $f \in \mathbb{V}_{\Katz',\mathbb{C}_p}^{\cusp, K_1^{(p)}(N)}$ is a prime-to-$Np$ Hecke eigenform and a $U_p$-eigenvector of eigenvalue $\alpha$ such that $|\alpha|=1$. Then, for $\rho$ the associated Galois representation and $\rho_p$ its restriction to a decomposition group, $\Gal(\overline{\mathbb{Q}}_p/\mathbb{Q}_p)$ admits an unramified subrepresentation $\rho_p$ on which geometric Frobenius acts by $\alpha$. 
\end{proposition}
\begin{proof}
This follows from the argument of \cite[\S6.13]{Kisin.OverconvergentModularFormsAndTheFontaineMazurConjecture}, noting (as in the proof of  \cref{lemma.classical-ordinary}) that the Galois representation used there is the dual of the Galois representation used here. 
\end{proof}

\subsubsection{Twists}\label{sss.twists-p-adic-mf}
We also note that there is a connected components map $\pi_0(\Ig^b)=\pi_0(\Ig_\Mant)=\mathbb{Z}_p^\times \times (\hat{\mathbb{Z}}^{(p)}(1))^\times$. In particular, for $\kappa$ any continuous character of $\mathbb{Z}_p^\times$, we can view $\kappa$ as a function on $\Ig_\Mant$ or $\Ig_b$ (cf. \cref{eq.Vmant-as-functions}). For $f \in \mathbb{V}_b$, we write $f \otimes \kappa$ for the product of $\kappa$ and $f$; this preserves $\mathbb{V}_{\Mant}$ (but not $\mathbb{V}_{\Katz'}$ or $\mathbb{V}_{\Katz}$!).  When $\kappa$ is a finite order character this is compatible with the classical twisting described in \cref{sss.mf-twists}: For $f \in S_k$, we have
\[ \eval_k(f \otimes \kappa) = \eval_k(f) \otimes \kappa. \]
One can also define twists by characters of $\hat{\mathbb{Z}}^{(p)}(1)$ similarly to \cref{sss.mf-twists}, but we need only the $p$-adic characters for our arguments below.

\subsection{The embedding and completion of the smooth Kirillov model}\label{ss.embedding-and-first-lower-bound}
Let $\pi=\pi_p \otimes \pi^{(p)}$ be an irreducible $\GL_2(\mathbb{A}_f)=\GL_2(\mathbb{Q}_p) \times \GL_2(\mathbb{A}_f^{(p)})$ subrepresentation appearing in $S_{k, \mathbb{C}_p}$, $k \geq 2$. We write $W_{\pi}:=  \Hom_{\GL_2(\mathbb{A}_f^{(p)})}( \pi^{(p)}, \mathbb{V}_{b,\mathbb{C}_p}^\cusp).$ 

\begin{lemma}\label{lemma.central-character}
    The action of the diagonal $\mathbb{Q}_p^\times$ in $T(\mathbb{Q}_p)$ on $W_\pi$ is by the central character of $\pi_p$ times $z \mapsto z^{-k}$. 
\end{lemma}
\begin{proof}
Since the action is continuous, it suffices to establish the identity for $a \in \mathbb{Q}^\times \subseteq \mathbb{Q}_p^\times$. By considering the multiplication by $a$ isogeny, we find this is equivalent to the action of $a^{-1}$ viewed as an element of the central $(\mathbb{A}_f^{(p)})^\times$, which is, by construction, through the central character of $\pi^{(p)}$. But the central character of $\pi^{(p)}$ on $\mathbb{Q}^\times$ is the inverse of the product of the central character of $\pi_p$ and the central character of the associated real discrete series $\pi_\infty$ (since $\pi_\infty \otimes \pi$, by construction as an automorphic representation, is trivial on the central $\mathbb{Q}^\times \leq \mathbb{A}_f^\times$). With our normalizations, the action on $\pi_\infty$ is by $r \mapsto r^{-k}$, thus we conclude. 
\end{proof}

\begin{remark}
    Note that the twist by $z^{-k}$ in the central action  matches the twist for $\eval_k$ appearing in \cref{ss.evaluation} (as it must by the equivariance described there!). 
\end{remark}

There is an $N>0$ and a new vector $v^{(p)} \in \pi^{(p)}$, unique up to scalar multiples, that is fixed by $K_1^{(p)}(N)$. The vector $v^{(p)}$ is a $U_\ell$ eigenvector for all $\ell | N$ and a $T_\ell$ eigenvector for all $\ell \not| pN$. Since $v$ generates $\pi^{(p)}$, we can identify $W_\pi$ with a closed subspace of $(\mathbb{V}_{b,\mathbb{C}_p}^\cusp)^{K_1^{(p)}(N)}$ by evaluation at $v^{(p)}$. If we then apply $\Kir$, we obtain an embedding
\[ W_\pi \hookrightarrow C^\bdd(\mathbb{A}_{f}^\times, \mathbb{C}_p). \]
Since $B_1(\hat{\mathbb{Z}}^{(p)}) \leq K^{(p)}_1(N)$, by the equivariance of $\Kir$, we find the embedding factors through the $B_1(\hat{\mathbb{Z}}^{(p)})$-invariants on the right. Writing 
\[ \mathbb{A}_f^\times = \mathbb{Q}_p^\times \times \mathbb{Z}_{(p), >0}^\times \times \hat{\mathbb{Z}}^{(p),\times}\]
and considering the action of $U(\hat{\mathbb{Z}}^{(p)})$, we find that the image lies in the functions supported on 
$\mathbb{Q}_p^\times \times \prod'_{\ell \neq p} \ell^{\mathbb{Z}_{\geq 0}} \times \hat{\mathbb{Z}}^{(p),\times}$. Considering the torus part of the action, we find the function does not depend on the last coordinate, so we obtain a function on $\mathbb{Q}_p^\times \times \prod'_{\ell \neq p}  \ell^{\mathbb{Z}_{\geq 0}}$. Having moreover fixed the $T_\ell$ and $U_\ell$ eigenvalues, by the usual computation of the effect of these operators on Fourier expansions, we find the composition with restriction to $\mathbb{Q}_p^\times \times 1$ is a topological embedding
\[ W_\pi \hookrightarrow C^\bdd(\mathbb{Q}_p^\times, \mathbb{C}_p). \]
Note that since $q$-expansions for elements of $\mathbb{V}_{b,\cusp}$ are, modulo $p^k$, power series in $q^{1/N}$ for some $N$, it follows from the definition of $\Kir$ (or, in a much more round-about-way, from \cref{main.completion-kirillov-model}) that this map factors as
\begin{equation}\label{eq.local-component-embedding} W_\pi \hookrightarrow C_\infty^{\bdd}(\mathbb{Q}_{p}^\times, \mathbb{C}_p). \end{equation}
where the right-hand side denotes functions that go to zero at $\infty$. 

Now, since $\pi_p \otimes \pi^{(p)} \subseteq S_{k, \mathbb{C}_p}$, the map $\eval_k$ of \cref{eq.eval-map-def} induces a map
\[ \pi_p \rightarrow W_\pi. \]
Composing with \cref{eq.local-component-embedding}, we obtain an embedding of $\pi_p$ into $C^{\bdd}(\mathbb{Q}_p^\times, \mathbb{C}_p)$. By the equivariance properties of $\eval_k$ and $\Kir$, the image $\pi_p^\Kir$ is the smooth Kirillov model of $\pi_p$ as in \cref{theorem.jaquetandkirillov}-(2). Because \cref{eq.local-component-embedding} is a topological embedding, we conclude the image of $W_\pi$ contains the closure $(\pi_p^\Kir)^\wedge$. 

\begin{lemma}\label{lemma.kir-completion}
    Under the embedding \cref{eq.local-component-embedding}, $W_\pi$ contains 
    \[ (\pi_p^\Kir)^\wedge= \mc{S}(\mathbb{Q}_p^\times, \mathbb{C}_p) + \mathbb{C}_p \cdot 1_{\mathbb{Z}_p} \cdot \kappa_{\ord}, \]
    where $\chi_\ord$ is the ordinary character appearing as a subrepresentation of $\rho_p$ if it exists (i.e. if $\pi$ is ordinary ---  see \cref{sss.ordinary-vector}) and $0$ otherwise.  
\end{lemma}
\begin{proof}
    It remains to compute $(\pi_p^\Kir)^\wedge$. The piece $\mc{S}(\mathbb{Q}_p^\times, \mathbb{C}_p)$ comes from taking the closure of $C_c^{\sm}(\mathbb{Q}_p^\times, \mathbb{C}_p)$, which is always contained in $\pi_p^\Kir$ by \cref{theorem.jaquetandkirillov}-(2)-(b). The ordinary character is obtained by \cref{lemma.classical-ordinary} using \cref{remark.canonical-lift-Kirillov-ordinary}. 
    
    It remains only to verify that no other vectors lie in the completion $\pi_p^\Kir$. This follows by observing that any other of the lines appearing in the presentations of the Kirillov models in \cref{example.kirillov-models} already lie in $\mc{S}(\mathbb{Q}_p^\times, \mathbb{C}_p)$. Indeed, for the functions $1_{\mathbb{Z}_p}\cdot \chi$ that appear, by construction of our embedding, they lie in $C^\bdd(\mathbb{Q}_p^\times, \mathbb{C}_p)$, thus the characters $\chi$ must satisfy $|\chi(p)| \leq 1$. If $|\chi(p)| < 1$, then this function lies in $\mc{S}(\mathbb{Q}_p^\times, \mathbb{C}_p)$, so it suffices to observe that there is at most one $\chi$ appearing with $|\chi(p)| = 1$. But this follows from the discussion of uniqueness of the ordinary line in \cref{sss.ordinary-vector}.
\end{proof}

\subsection{The upper bound}\label{ss.upper-bound}
Since $\mathcal{S}(\mathbb{Q}_p^\times, \mathbb{C}_p) \subseteq W_\pi \subseteq C^\bdd_{\infty}(\mathbb{Q}_p, \mathbb{C}_p)$ and 
\[ C^\bdd_\infty(\mathbb{Q}_p, \mathbb{C}_p)_{\tilde{\mu}_{p^\infty}}=C^\bdd_\infty(\mathbb{Q}_p, \mathbb{C}_p) / \mc{S},  \]
we find 
\[ W_{\pi}/\mathcal{S}(\mathbb{Q}_p^\times, \mathbb{C}_p) = W_{\pi, \tilde{\mu}_{p^\infty}}. \]

Since $\pi$ can be defined and $v^{(p)}$ chosen over a finite extension $L/\mathbb{Q}_p$, and the construction of $W_\pi$ and formation of $\tilde{\mu}_{p^\infty}$-coinvariants are compatible with base change, we may replace $\mathbb{C}_p$ everywhere with $L$ in our study of $W_{\pi, \tilde{\mu}_{p^\infty}}$.

Because the diagonal $\mathbb{Q}_p^\times \subseteq T(\mathbb{Q}_p)$ acts on $W_\pi$ and thus $W_{\pi, \tilde{\mu}_{p^\infty}}$ through a  character by \cref{lemma.central-character}, \cref{maincor.admissibility} implies that $W_{\tilde{\mu}_{p^\infty}}$ is an admissible $L$-Banach representation of $\mathbb{Q}_p^\times$ embedded as $\mathbb{Q}_p^\times \times 1$ thus, in particular, of $1+2p\mathbb{Z}_p$ --- i.e., it is the dual of a finitely generated module over 
\[ \mathcal{O}_L[[1+2p\mathbb{Z}_p]][1/p]=\mathcal{O}_L[[t]][1/p]. \]
Since this is a principal ideal domain, the structure theory of modules over a PID implies that, after possibly replacing $L$ with a larger extension, the $1+2p\mathbb{Z}_p$-module $W_{\pi, \tilde{\mu}_{p^\infty}}$ is a direct sum of a finite number of copies of $C(1+2p\mathbb{Z}_p, L)$ (corresponding to the free part of the dual) and a finite sum of modules supported at a single character and finite dimensional over $L$ (correponding to the torsion of the dual). 

For each character space that appears in $W_{\pi, \tilde{\mu}_{p^\infty}}$ (which is all of them if it contains a copy of $C(1+2p\mathbb{Z}_p, L)$), we can decompose the associated finite dimensional character space for the action of $1+2p\mathbb{Z}_p$ according to characters of $\mu(\mathbb{Q}_p)$, and, up to enlarging $L$, further decompose into generalized eigenspaces for the action of $p \in \mathbb{Q}_p^\times$. On each such eigenspace, the eigenvalue $\alpha$ must be a unit since $\mathbb{Q}_p^\times$ acts unitarily (the action of $T(\mathbb{Q}_p)$ is unitary already since it preserves $\mathbb{V}_b^\cusp$). 

For $\chi$ a character of $\mathbb{Q}_p^\times$ that appears, let $\kappa =\chi|_{\mathbb{Z}_p^\times}$ so that $\chi(p^ku)=\alpha^k \kappa(u)$. We may take the associated ordinary form $w \in \mathbb{V}_{\Mant}^{\cusp, K_1^{(p)}(N)}$ using \cref{main.hida-coinvariants}. Then, the twist $w'=w\otimes \kappa^{-1} \in \mathbb{V}_{\Katz'}^{\cusp, K_1^{(p)}(N)}$ is an ordinary eigenform whose associated Galois representation is the twist $\rho \otimes \kappa^{-1}$ (where $\kappa^{-1}$ is composed with the $p$-adic cyclotomic character). From \cref{prop.ordinary-p-adicmf-unramified-sub} we deduce that $\ur_{\alpha}$, the unramified character sending $p$ to $\alpha$, is a subrepresentation of $\rho_p \otimes \kappa^{-1}$, and thus that $\chi$ is a subrepresentation of $\rho_p$. 

In particular, since $\rho_p$ contains at most two characters, we deduce that $W_{\tilde{\mu}_{p^\infty}}$ cannot contain any copies of $C(1+2p\mathbb{Z}_p, L)$ (as otherwise there would be infinitely many distinct characters appearing in the above argument). Writing these characters as $\chi_i$, from the admisibility of $W_{\pi, \tilde{\mu}_{p^\infty}}$ it then follows that each vector can be decomposed into a sum of generalized eigenvectors for the $\chi_i|_{1+2p\mathbb{Z}_p}$ (note that these characters are distinct if there are two different $i$ because of the Hodge-Tate weights). Fixing a $\chi=\chi_i$, any such generalized eigenvector is then an eigenvector for $\chi|_{\mu(\mathbb{Q}_p)}$, and is in the $\alpha=\chi(p)$-generalized eigenspace for the action of $p$. 

In particular, by \cref{theorem.ordinary-iso-coinvariants}, such a vector is represented by a generalized $\alpha$-eigenvector for $U_p$ acting on $C(\mathbb{Z}_p\backslash\{0\}, \mathbb{C}_p)$. If $(U_p - \alpha)^m f=0$, then we can write $f(p^k u)=f_1 (u) \alpha^k + f_2(u) \alpha^k k + \ldots f_m(u)\alpha^k k^{m-1}$ where $f_i(u) \in C(\mathbb{Z}_p^\times, \mathbb{C}_p)$. Writing $\mathbb{Z}_p^\times=\mu(\mathbb{Q}_p) \times (1+2p\mathbb{Z}_p)$, each $f_i(u)$ must be a generalized $\chi|_{1+2p\mathbb{Z}_p}$-eigenvector and a $\chi|_{\mu(\mathbb{Q}_p)}$-eigenvector, thus, writing $u=(\zeta,t)$, $f_i(u)=\chi(\zeta) (\sum a_{i,j} \chi(t)\log(t)^j)$ for constants $a_{i,j} \in \mathbb{C}_p$. This gives the description of the upper bound in \cref{main.lg}. 

\subsection{The lower bound}\label{ss.lower-bound-companion}

Let $\chi$ be a character appearing in $\rho_p$. Since the Hodge-Tate weights of $\rho_p$ are $0$ and $k-1$, $\Lie \chi = 0$ or $\Lie \chi=1-k$. 

Suppose first that $\Lie \chi=0$. Then $\WD(\chi)=\chi \subseteq \sigma_p = \WD(\rho_p)$, thus $\chi \subseteq J(\pi_p)$, $\pi$ is ordinary, and $\chi=\chi_\ord$. In this case, we have already established $1_{\mathbb{Z}_p} \cdot \chi \in W_\pi$ in \cref{lemma.kir-completion}. 

On the other hand, if $\Lie \kappa=1-k$, then the quotient $\chi'$ of $\rho_p$ by $\chi$ has Hodge-Tate weight $0$, so $\chi'=\WD(\chi')$ is ordinary. As in the proof of \cref{lemma.classical-ordinary}, $\chi'$ is also a subrepresentation of $\rho_p$. Thus $\rho_p$ is split. Twisting, we may assume $\chi'$ is unramified, so that the vector $1_{\mathbb{Z}_p} \cdot \chi' \in W_\pi$ corresponds to a classical newform $f$ (recall that we have already arranged in our choice of $v^{(p)}$ to be new away from $p$). In this case, let $\kappa=\WD(\chi)|_{\mathbb{Z}_p^\times}$, let $\tilde{\chi}(z)=\WD(\chi) \kappa^{-1}(z|z|)$, and let $\tilde{\pi} = \pi \otimes \kappa^{-1}$. Then the function $1_{\mathbb{Z}_p} \cdot \tilde{\chi}$ lies in the Kirillov model of $\tilde{\pi}$  and corresponds to the classical newform $\tilde{f}$ of \cite[Theorem 1.1.3]{BreuilEmerton.RepresentationPAdiquesOrdinarsDeGl2Qp}. In particular, since $\rho_p$ is split, \cite[Theorem 1.1.3]{BreuilEmerton.RepresentationPAdiquesOrdinarsDeGl2Qp} implies 
$\tilde{f}=\theta^{k-1}(g)$ for $g$ an ordinary $p$-adic modular form of weight $2-k$. We note that, as in \cite{Howe.AUnipotentCircleActionOnpAdicModularForms}, on $\mathbb{V}_{\Katz'}$, $\theta^{k-1}$ is given by  acting by $z^{k-1} \in C(\mathbb{Z}_p, \mathbb{C}_p)$ in $\mathbb{V}_{\Mant}$ (which does not preserve $\mathbb{V}_{\Katz'}$) followed by twisting by $z^{1-k}$ in $C(\mathbb{Z}_p^\times, \mathbb{C}_p)$ (which returns to $\mathbb{V}_{\Katz'}$). In particular, if we consider the twist $\tilde{g}:=g \otimes z^{1-k}$, then $\tilde{g}$ lies in $W_{\tilde{\pi}}$ and $z^{k-1} \cdot \tilde{g} = \tilde{f}$. In particular, $(\tilde{g}\otimes \kappa)$ lies in $W_{\pi}$, and $z^{k-1} \cdot (\tilde{g} \otimes \kappa) = \tilde{f} \otimes \kappa$, which, in the Kirillov model, is $1_{\mathbb{Z}_p} \cdot \tilde{\chi}\kappa(z|z|) = 1_{\mathbb{Z}_p} \cdot \WD(\chi)$. Since $\chi$ has Hodge-Tate weight $k-1$, $\WD(\chi)=\chi z^{k-1}$, thus  $\tilde{g}\otimes \kappa$ gives $1_{\mathbb{Z}_p} \cdot \chi$ in the Kirillov model.

\section{A conjectural Hida theory and admissible local-global compatibility for general Igusa varieties}\label{s.conjecture}

After some preliminary setup in \cref{ss.conj-setup}, the first main contribution of this section is a precise finiteness conjecture generalizing \cref{maincor.admissibility} to more general Igusa varieties (\cref{conj.precise-coinvariants}). In \cref{ss.model-representations} we then explain a natural geometric construction of representations that, in the case of the $\GL_2$ ordinary Caraini-Scholze Igusa variety, recovers the characters appearing in \cref{main.lg} as the dual of a space of bounded sections on a (quotient of) the associated local Shimura variety (\cref{example.model-construction}). This construction comes with a mechanism, via Rapoport-Zink/Caraiani-Scholze uniformization, for relating it directly to functions on Igusa varieties. Based on this construction and \cref{main.lg}, in \cref{ss.conj-local-global-compatibility} we speculate on generalizations of the admissible local-global compatibility to the setting of \cref{conj.precise-coinvariants}. 

\subsection{Setup}\label{ss.conj-setup}
Let $(G,X)$ be a PEL Shimura datum of type A/C with $G$ anisotropic mod center, let $\mu$ be an associated Hodge cocharacter, and let $p$ be a prime where $G$ is unramified (i.e. $G_{\mathbb{Q}_p}$ is quasisplit and split over an unramified extension). Let $b \in B(G_{\mathbb{Q}_p},\mu^{-1})$, the Kottwitz set of \cite{Kottwitz.IsocrystalsWithAdditionalStructure}. We fix an embedding of the reflex field into $\Qpbreve$, which exists by our unramified assumption. 

For a compact open subgroup $K^p\leq G(\mathbb{A}_f^{(p)})$ we let $\Ig^b_{K^p}$ be the associated Caraiani-Scholze formal scheme (denoted $\mathrm{Ig}^b_{\mathcal{O}_{\breve{E}}}$ in \cite[bottom of p. 718]{CaraianiScholze.OnTheGenericPartOfTheCohomologyOfCompactUnitaryShimuraVarieties}). By its construction and the argument of \cite[Proposition 3.3.4] {CaraianiScholze.OnTheGenericPartOfTheCohomologyOfNonCompactUnitaryShimuraVarieties}, $\Ig_{K^p}^b$ is an affine $p$-adic formal scheme --- it is represented by the Witt vectors of its special fiber $\mathrm{Ig}_{K^p}^b=\Ig^b_{K^p, \overline{\mathbb{F}}_p}$, a perfect affine scheme. We write $\Ig^b=\lim_{K^p} \Ig_{K^p}^b$, and it follows that $\Ig^b$ is also an affine $p$-adic formal scheme. It admits a natural action of $\Aut_{G}(\tilde{\mbb{X}}_b) \times \GL_2(\mathbb{A}_f^{(p)})$, where the first group is as defined in \cite[Definition 4.2.9]{CaraianiScholze.OnTheGenericPartOfTheCohomologyOfCompactUnitaryShimuraVarieties}. We also write $\mathrm{Ig}^b = \Ig^b_{ \overline{\mathbb{F}}_p}.$

\begin{remark}\label{remark.compact-setup}
 For simplicity, we have set up our hypotheses so that the associated Shimura varieties are proper. In particular, in contrast to our main results, we do not need to pass to a cuspidal part of $\mathcal{O}(\Ig^b)$ below. For more general Shimura varieties, to define good spaces of $p$-adic automorphic forms one should instead use functions on a partial compactification of $\Ig^b$ as in \cite{CaraianiScholze.OnTheGenericPartOfTheCohomologyOfNonCompactUnitaryShimuraVarieties}; note, however, that even in the non-compact case many of the Igusa varieties will not interact with the boundary thus will not need a partial compactification (for example, the basic Igusa variety never needs to be compactified). The conjectures and discussion below can thus be extended as-is to these cases.   
\end{remark}

\subsection{Coinvariants}\label{ss.coinvariants}
Let $\tilde{U}_b$ be the connected component of the identity in $\Aut_{G}(\tilde{\mbb{X}}_b)$, and let $G_b$ be the automorphism group of the $G$-isocrystal associated to $b$ (a form of $M_{\nu}$, the centralizer of the slope morphism $\nu$ of $b$). Then there is a natural decomposition
\[ \Aut_{G}(\tilde{\mbb{X}}_b) = \ul{G_b(\mathbb{Q}_p)} \ltimes \tilde{U}_b.\]

By \cite[Proposition 4.2.11]{CaraianiScholze.OnTheGenericPartOfTheCohomologyOfCompactUnitaryShimuraVarieties}, $\tilde{U}_b$  is isomorphic to $\Spf R$ where 
\[ R \cong \Zpbreve[[x_1^{1/p^\infty}, \ldots, x_n^{1/p^\infty}]] \]
is equipped with the $(p,x_1, \ldots, x_n)$-adic topology, and the identity automorphism corresponds to the $\Zpbreve$ point given by setting $x_i=0$, $i = 1, \ldots, n$. In particular, 
\[ \tilde{U}_b \times_{\Spf \Zpbreve} \tilde{U}_b = \Spf (R \hat{\otimes}_{\Zpbreve} R). \] 
We write $R^*$ for the continuous $\Zpbreve$-linear homomorphisms from $R$ to $\Zpbreve$. We claim that the group structure $\tilde{U}_b$ induces an algebra structure on $R^*$. Indeed, the group law is given by a  topological comultiplication $R \rightarrow R \hat{\otimes}_{\Zpbreve} R$, which induces a (not necessarily commutative) multiplication through the composition 
\[ R^* \otimes_{\Zpbreve} R^* \rightarrow (R \hat{\otimes}_{\Zpbreve} R)^* \rightarrow R^*.\]
 The inclusion of the constants $\epsilon:\Zpbreve \rightarrow R$ defining the $\Zpbreve$-algebra structure on $R$ induces an augmentation homomorphism $\epsilon^*: R^* \rightarrow \Zpbreve$ whose kernel we write as $I$. 

Suppose given an affine $p$-adic formal scheme $\Spf V/\Spf \Zpbreve$ with an action of $\tilde{U}_b$. Then the action map $\tilde{U}_b \times_{\Spf \Zpbreve} \Spf V \rightarrow \Spf V$ is induced by a map 
\[ V \rightarrow V \hat{\otimes}_{\Zpbreve} R= V[[x_1^{1/p^\infty},\ldots,x_n^{1/p^\infty}]]. \]
In particular, we  obtain an algebra action of $R^*$ on $V$ by contraction with this map. We then define the topological coinvariants by 
\[ V_{\tilde{U}_b} := V \hat{\otimes}_{R^*} R^*/I := \lim_{n} V/p^n \otimes_{R^*} R^*/(I,p^n) = \lim_n V_{\Zpbreve/p^n} / I_{\Zpbreve/p^n} \cdot V_{\Zpbreve/p^n}. \]

If $\Spf V$ is equipped with an action of $\Aut_G(\tilde{\mathbb{X}}_b)$ extending the action of $\tilde{U}_b$, then there is a residual action of $G_b(\mathbb{Q}_p)$ on the topological co-invariants --- indeed, $G_b(\mathbb{Q}_p)$ acts continuously on each term and preserves $I$. Using the natural action of $\Aut_{G}(\tilde{\mathbb{X}}_b)$ on $\Ig^b_{K^p}$, we can now formulate our conjecture: 

\begin{conjecture}\label{conj.precise-coinvariants}With notation as above, for $K^p \leq G(\mathbb{A}_f^{(p)})$ a compact open subgroup, the $p$-inverted topological coinvariants 
\[ \mathcal{O}(\Ig^b_{K^p})_{\tilde{U}_b}[1/p] \]
are an admissible Banach $G_b(\mathbb{Q}_p)$-representation. Equivalently, 
\[ \mathcal{O}(\mathrm{Ig}^b_{K^p})_{\tilde{U}_b} = \mathcal{O}(\mathrm{Ig}^b_{K^p}) / I_{\overline{\mathbb{F}}_p} \cdot \mathcal{O}(\mathrm{Ig}^b_{K^p})  \]
is an admissible\footnote{Since $G_b(\mathbb{Q}_p)$ is locally pro-$p$, admissibility of a smooth representation of $G_b(\mathbb{Q}_p)$ on a $\overline{\mathbb{F}}_p$-vector space can be checked on a single compact open pro-$p$ group, i.e. it suffices to know to that for any single compact open pro-$p$ subgroup $K_b \leq G_b(\mathbb{Q}_p)$, the $K_b$-invariants are a finite dimensional $\overline{\mathbb{F}}_p$-vector space.} smooth $G_b(\mathbb{Q}_p)$-representation over $\overline{\mathbb{F}}_p$.   	
\end{conjecture}

\begin{example}\cref{maincor.admissibility} establishes the analog of \cref{conj.precise-coinvariants} for $(\GL_2, \mathbb{H}^\pm)$ and $b=\begin{bmatrix}p^{-1} & 0 \\ 0 & 1\end{bmatrix}$, i.e.  for the Igusa variety attached to the ordinary locus on the modular curve (where we study cuspidal functions on a partial compactification, which are not, strictly speaking, covered by the setup; cf. \cref{remark.compact-setup}). 
\end{example}

 \begin{remark}\label{remark.mu-ord-and-basic}
 When $b$ is ($\mu$-)ordinary, corresponding to the open ($\mathrm{mod}\, p$) Newton stratum, \cref{conj.precise-coinvariants} can sometimes be related to existing generalizations of Hida theory as in the proof of \cref{maincor.admissibility}. When $b$ is basic, corresponding to the  closed ($\mathrm{mod}\, p$) Newton stratum, $\Ig_b$ is zero dimensional and $\tilde{U}_b$ is trivial, so the conjecture is just that $\mathbb{V}_{b,\Qpbreve}^{K^p}$ is admissible. However, in this case $\mathbb{V}_{b,\Qpbreve}^{K^p}$ can be described explicity as a finite product of function spaces of the form 
 \[ \Cont(G'(\mathbb{Q})\backslash G'(\mathbb{A}_f)/K^p, \Qpbreve) \]
 where in each term $G'$ is an inner form of $G$ that is compact modulo center at the real place and such that $G'_{\mathbb{Q}_p}=G_b$ (see, e.g., \cite{Howe.ThepAdicJacquetLanglandsCorrespondenceAndAQuestionOfSerre} for a detailed account of the example of $G=\GL_2$, in which case there is only a single term and $G'$ is the group of units in the quaternion algebra over $\mathbb{Q}_p$ ramified at $p$ and $\infty$, and \cite[Corollary 11.28]{Zhang.APelTypeIgusaStackAndThePAdicGeometryOfShimuraVarieties} in general);  the admissibility is then immediate. To our knowledge, the conjecture for any intermediary Newton stratum  is completely open. 
 \end{remark}

\subsection{Model representations}\label{ss.model-representations}

We consider the infinite level local Shimura variety $\mathcal{M}_{b,[\mu]}$ over $\mathbb{C}_p$ as in \cite{Scholze.pAdicGeometry} (we use the notation as in \cite{HoweKlevdal.AdmissiblePairsAndpAdicHodgeStructuresIITheBiAnalyticAxLindemannTheorem}) and its $v$-sheaf quotient 
\[ \mathcal{M}_{b,[\mu]}':=\mathcal{M}_{b,[\mu]}/\tilde{U}_b. \]
The latter is a $G_b(\mathbb{Q}_p)$-torsor over the Newton stratum $\Fl_{[\mu],[b]}$, a locally closed subdiamond of the the rigid analytic flag variety $\Fl_{[\mu]}$ over $\mathbb{C}_p$. 

\subsubsection{} For any $G$-equivariant vector bundle $\mathcal{V}$ on $\Fl_{[\mu]}$, there is a natural notion of bounded sections over $\mathcal{M}_{b,[\mu]}'$,
\[ H^{0,\bdd}(\mathcal{M}_{b,[\mu]}', \mathcal{V}). \]
Indeed, any such bundle $\mc{V}$ admits an integral structure $\mathcal{V}^+$ and, by quasi-compactness of the flag variety, any two integral structures $\mathcal{V}^+_1$ and $\mathcal{V}^+_2$ are commensurate (i.e., for $n \gg 0$,  $p^n \mathcal{V}^+_1 \subseteq \mathcal{V}^+_2 \subseteq p^{-n} \mathcal{V}^+_1$). The action of $G(\mathbb{Q}_p)$-permutes these integral structures, thus we obtain a notion of bounded sections preserved by $G(\mathbb{Q}_p)$ by taking the colimit of global sections over all such integral structures. We topologize it as follows: for any integral structure $\mathcal{V}^+$, 
\[  H^{0,\bdd}(\mathcal{M}_{b,[\mu]}', \mathcal{V})= H^{0}(\mathcal{M}_{b,[\mu]}', \mathcal{V}^+)[1/p], \]
and $H^{0}(\mathcal{M}_{b,[\mu]}', \mathcal{V}^+)$ admits a natural topology as the inverse limit of the $p$-adic topologies on sections over quasi-compact opens. This topology is independent of the integral structure; we note that it is weaker than the Banach topology obtained by taking the the global sections of any integral structure with its $p$-adic topology as the unit ball. The action of $G(\mathbb{Q}_p)$ is continuous but does not typically preserve an integral lattice; by contrast, the $G_b(\mathbb{Q}_p)$-action preserves the integral lattices.  

Note that, for $\mc{S}=\lim_{K^p} \mc{S}_{K^p, \mathbb{C}_p}$ and $\pi_\HT: \mathcal{S} \rightarrow \Fl_{[\mu]}$ the Hodge-Tate period map, $\pi_\HT^* \mc{V}$ is a traditional automorphic vector bundle.
If $\bigotimes \pi_\ell$ is an irreducible smooth representation of $G(\mathbb{A}_f)$ appearing in $H^0(\mathcal{S}_{\mathbb{C}_p}, \pi_\HT^* \mc{V})$, then we expect the uniformization of the $b$-Newton stratum to yield a map of $G_b(\mathbb{Q}_p)$-representations
\begin{equation}\label{eq.homs} \Hom_{G(\mathbb{Q}_p)}(\pi_p, H^{0,\bdd}(\mathcal{M}_{b,[\mu]}', \mc{V}))^* \rightarrow \Hom_{G(\mathbb{A}_f^{(p)})}(\otimes_{\ell \neq p}\pi_\ell,  \mathcal{O}(\Ig^{b})_{\tilde{U}_b}).\end{equation}
We hope to return to the construction of this map in another work\footnote{The idea is as follows: first, all global sections of $\mc{V}$ over $\mc{S}$ are bounded since $\mc{S}$ is quasi-compact. Now, the uniformization of \cite[Lemma 4.3.20]{CaraianiScholze.OnTheGenericPartOfTheCohomologyOfCompactUnitaryShimuraVarieties} gives a map $\Ig_{b,\eta, \mathbb{C}_p} \times \mathcal{M}_{b,[\mu]} \rightarrow \mc{S}$, matching the Hodge-Tate period map on $\mc{S}$ with that on $\mathcal{M}_{b,[\mu]}$. In particular, the pullback of bounded sections on $\mc{S}$ land in bounded sections on $\Ig_{b,\eta, \mathbb{C}_p} \times \mc{M}_{b,[\mu]}$, which we expect to be given by a base change formula as $\mathcal{O}(\Ig^b) \hat{\otimes} H^{0,\bdd}(\mc{M}_{b,[\mu]}, \mc{V})$. The $G(\mathbb{Q}_p)$-action is concentrated on the local Shimura variety and the $G(\mathbb{A}_f^{(p)})$ action is concentrated on $\Ig^b$. Thus by pulling back $\bigotimes \pi_\ell \subseteq H^0(\mc{S}, \pi_\HT^* \mc{V})$ along the uniformization map, we should obtain a class in
\[ \Hom_{G(\mathbb{A}_f^{(p)})} (\otimes'_{\ell \neq p} \pi_\ell, \mathcal{O}(\Ig^b)[1/p] \hat{\otimes} \mathbb{C}_p) \hat{\otimes}_{\mathbb{C}_p} \Hom_{G(\mathbb{Q}_p)}(\pi_p, H^{0,\bdd}(\mc{M}_{b,[\mu]}, \mc{V})). \]
By construction of the uniformization map, this class is preserved by the diagonal $\tilde{G}_b$ action, thus it induces, by contraction with the right hand side, a $\tilde{G}_b$-equivariant map 
\[ \Hom_{G(\mathbb{Q}_p)}(\pi_p, H^{0,\bdd}(\mc{M}_{b,[\mu]}, \mc{V}))^* \rightarrow \Hom_{G(\mathbb{A}_f^{(p)})} (\otimes'_{\ell \neq p} \pi_\ell, \mathcal{O}(\Ig^b)[1/p]\hat{\otimes}{\mathbb{C}_p}). \]
Passing to the $\tilde{U}_b$-coinvariants and identifying the coinvariants of the dual on the left with the dual of the $\tilde{U}_b$-invariants $\Hom(\pi_p, H^{0,\bdd}(\mc{M}'_{b,[\mu]}, \mc{V}))$, one should obtain \cref{eq.homs}. There are several details to check to verify such a construction proceeds as claimed! 
}; here we wish only to understand what it suggests about local global compatibility for $\mathcal{O}(\Ig^{b})_{\tilde{U}_b}$. 

\begin{example}\label{example.model-construction}
    In the the case studied in \cref{main.lg}, there is a natural identification of $\mathcal{M}_{b,[\mu]}'$ with $U(\mathbb{Q}_p)\backslash \GL_2(\mathbb{Q}_p)$, i.e. the $\mathbb{Q}_p$-points of the base affine bundle over $\mathbb{P}^1$ --- this follows by considering the transitive action of $G(\mathbb{Q}_p)$ on $\mathcal{M}'_{b,[\mu]}$.  
    
    The global sections of the usual bundle of weight $(k_1, k_2)$ over $\mathcal{M}_{b,[\mu]}'$ can be written as 
    \[ C(U(\mathbb{Q}_p)\backslash \GL_2(\mathbb{Q}_p), \Qpbreve), \] 
    where the action of $T(\mathbb{Q}_p)=G_b(\mathbb{Q}_p)$ is by 
    \[ ((z_1, z_2)\cdot f)(x)=z_1^{k_1} z_2^{k_2} f((z_1^{-1}, z_2^{-1})x)\]
    In this presentation, the bounded sections are those functions such that $(g, (z_1,z_2))\mapsto f((z_1, z_2)g)z_1^{k_1}z_2^{k_2}$ is a bounded function on $\GL_2(\mathbb{Z}_p) \times T(\mathbb{Q}_p)$. 
    
    In particular, if we consider a smooth character $\epsilon$ of $T(\mathbb{Q}_p)$ and view the continuous (unnormalized) induction $\mathrm{Ind}_{B(\mathbb{Q}_p)}^{\GL_2(\mathbb{Q}_p)}(\epsilon)$ as the set of continuous functions on $U(\mathbb{Q}_p)\backslash \GL_2(\mathbb{\mathbb{Q}}_p)$ such that $f(tx)=\epsilon(t)f(x)$, then under this identification they correspond to bounded sections if and only if $\epsilon(z_1, z_2)=|z_1|^{-k_1}|z_2|^{-k_2} \cdot \epsilon^\circ$ where $\epsilon^\circ$ is a unitary character. 
    
    Thus, for the sheaf $\mathcal{V}_{(0,k)}$ that pulls back via $\pi_\HT$ to $\omega_{E^\vee}^k$, we find it is exactly the inductions of $\epsilon=|z_2|^{-k} \cdot \epsilon^\circ$ that live in the bounded sections. Note that the smooth vectors are precisely the $\pi_p$ appearing in ordinary representations $\pi$ as in \cref{sss.ordinary-vector}. For $\pi_p$ the smooth induction of such a character, the $G_b(\mathbb{Q}_p)=T(\mathbb{Q}_p)$-action on $\Hom(\pi_p, (H^{0,\bdd}(\mathcal{M}_{b,[\mu]}', \mathcal{V}_{0,k}))^*$ is by the character $z_2^{-k} \epsilon$. This is precisely the action corresponding to the ordinary character in \cref{conj.lg-mf}/\cref{main.lg} and \cref{remark.central-action}/\cref{lemma.central-character}.  
\end{example}

\subsection{Local-global compatibility}\label{ss.conj-local-global-compatibility}
One can thus hope\footnote{Emphasis on ```hope": what follows is very speculative!} that the admissible part of the local-global compatibility in \cref{main.lg} will generalize in the following way: suppose given $\bigotimes_{\ell}\pi_\ell$ appearing in $H^0(\mathcal{S}, \pi_\HT^* \mc{V})^\sm$ for $\mc{V}$ sufficiently regular. Then, we expect that there is an associated Galois representation $\rho: \Gal(\overline{\mathbb{Q}}/\mathbb{Q}) \rightarrow {}^{C}G(\mathbb{C}_p)$ as in \cite{BuzzardGee.TheConjecturalConnectionsBetweenAutomorphicRepresentationsAndGaloisRepresentations} and that 
\[ \Hom(\pi_p, H^{0,\bdd}(\mathcal{M}_{b,[\mu]}', \mc{V})) \]
is non-zero if and only if the $\rho_p=\rho|_{\Gal(\overline{\mathbb{Q}}_p/\mathbb{Q}_p)}$ factors through the conjugacy class $P_{\nu}$ of parabolics of ${}^C G$ determined by the slope morphism $\nu$  associated to $b$. In the non-zero case, we expect the associated map (\ref{eq.homs}) to factor through an admissible Banach quotient of 
\[ \Hom_{G(\mathbb{Q}_p)}(\pi_p, H^{0,\bdd}(\mathcal{M}_{b,[\mu]}', \mc{V}))^* \]
depending, up to finite indeterminacy, only on $\rho_p^\mathrm{ss}$, the push-out of $\rho_p$ to $M_\nu(\mathbb{C}_p)$, for $M_\nu$ the Levi quotient of $P_\nu$. When $\rho_p$ is maximally non-split, we might hope 
\begin{equation}\label{eq.hom-igb-equation} \Hom_{G(\mathbb{A}_f^{(p)})}(\otimes_{\ell \neq p}\pi_\ell, \mathcal{O}(\Ig_b)_{\tilde{U}_b}[1/p] \hat{\otimes} \mathbb{C}_p) \end{equation}
is the topological direct sum of these subrepresentations as we vary over all irreducible subrepresentations of $H^0(\mathcal{S}, \pi_\HT^* \mc{V})^\sm$ with the same prime-to-$p$ factors. 

\begin{remark}
    When $\rho_p$ is not maximally non-split, \cref{main.lg} and its proof suggest that, to see more of the representation \cref{eq.hom-igb-equation}, one should consider maps analogous to \cref{eq.homs} coming not just from global sections but also from more general overconvergent sections. 
\end{remark}

\begin{remark}
    One hopes to have similar statements for all Hecke eigensystems appearing in $\mathcal{O}(\Ig_b)[1/p]$, not just those attached to classical forms. 
\end{remark}

\begin{remark}
In \cref{conj.precise-coinvariants}, we also conjectured an equivalent mod $p$ admissibility. One can thus hope to have a mod $p$ local-global compatibility. However, since the representations of $G(\mathbb{Q}_p)$ in the local construction of \cref{ss.model-representations} do not typically preserve a lattice, we do not have any clear expectations for such a statement.
\end{remark}    

\begin{remark}
    One can also hope to understand the situation before passing to $\tilde{U}_b$-coinvariants by working with bounded sections over $\mathcal{M}_{b,[\mu]}$ instead of $\mathcal{M}_{b,[\mu]}'$, but the analysis will be more difficult due to the lack of any obvious reasonable finiteness conditions. The form of \cref{main.lg} suggests that, in any case, the most interesting information about $\rho_p$ will be that which is encoded in the coinvariants. 
\end{remark}

\bibliographystyle{plain}
\bibliography{references, preprints}

\end{document}